\documentclass[a4paper,12pt]{article}

\usepackage{amscd,amssymb,amsmath,amsfonts,amsthm}
\usepackage{graphicx}
\usepackage{verbatim,enumerate,paralist}
\usepackage{textcomp}
\usepackage[colorlinks]{hyperref}
\hypersetup{%
  urlcolor=blue,linkcolor=blue,citecolor=blue
}
\usepackage{pdfpages}

\usepackage[T1]{fontenc}
\usepackage[all,poly,knot,tips]{xy}
\newcounter{Definitioncount}
\newtheorem*{Theorem}{Theorem}

\newtheorem*{Proposition}{Proposition}

\newtheorem*{Corollary}{Corollary}

\newtheorem*{Conjecture}{Conjecture}
\newtheorem*{Lemma}{Lemma}

\theoremstyle{definition}
\newtheorem*{Definition}{Definition}
\newtheorem*{Observation}{Observation}
\newtheorem*{Observations}{Observations}
\newtheorem*{Remark}{Remark}
\newtheorem*{Example}{Example}
\newtheorem{example}{Example}[subsection]

\newtheoremstyle{fact}{\bigskipamount}{\medskipamount}{\upshape}{}{\itshape}{. }{ }{Fact}
\theoremstyle{fact}
\newtheorem*{Fact}{Fact}

\newtheoremstyle{genquest}{\bigskipamount}{\medskipamount}{\upshape}{}{\itshape}{. }{ }{General Question}
\theoremstyle{genquest}
\newtheorem*{GenQuest}{General Question}

\newtheoremstyle{step}{2\bigskipamount}{\medskipamount}{\upshape}{}{\itshape}{. }{ }{\underline{Step~\thestep}}
\theoremstyle{step}
\newtheorem{step}{Step}[subsection]
\renewcommand{\thestep}{\arabic{step}}

\newcommand{\catI}{\mathbf{1}}
\newcommand{\catA}{\mathcal{A}}
\newcommand{\Ab}{\operatorname{Ab}}
\newcommand{\Aut}{\operatorname{Aut}}
\newcommand{\catB}{\mathcal{B}}
\newcommand{\classB}{\operatorname{B}}
\newcommand{\BB}{\mathbb{B}}
\newcommand{\BC}{\classB C}
\newcommand{\BStMon}{\operatorname{BStMon}}
\newcommand{\Cat}{\operatorname{Cat}}
\newcommand{\catC}{\mathcal{C}}
\newcommand{\CC}{\mathbb{C}}
\newcommand{\funcC}{\mathrm{C}}
\newcommand{\CP}{\operatorname{CP}}
\newcommand{\catD}{\mathcal{D}}
\newcommand{\Drin}{\mathrm{D}}
\newcommand{\LieD}{\operatorname{D}\nolimits}
\newcommand{\Ds}{\operatorname{Ds}}
\newcommand{\catE}{\mathcal{E}}
\newcommand{\univE}{\operatorname{E}}
\newcommand{\catF}{\mathcal{F}}
\newcommand{\FF}{\mathbb{F}}
\newcommand{\Ext}{\operatorname{Ext}}
\newcommand{\FCat}{\operatorname{\mathcal{F}\mbox{-}Cat}}
\newcommand{\catG}{\mathcal{G}}
\newcommand{\GL}{\operatorname{GL}}
\newcommand{\catGL}{\mathop{\mathcal{G}\ell}\nolimits}
\newcommand{\GSet}{\operatorname{G-set}}
\newcommand{\catH}{\mathcal{H}}
\newcommand{\Hall}{\operatorname{Hall}}
\newcommand{\coHom}{\operatorname{H}}
\newcommand{\scrI}{\mathcal{I}}
\newcommand{\Ind}{\operatorname{Ind}}
\newcommand{\ringK}{\mathrm{K}}
\newcommand{\Lan}{\operatorname{Lan}}
\newcommand{\MM}{\mathbb{M}}
\newcommand{\funcM}{\mathcal{M}}
\newcommand{\Mky}{\operatorname{Mky}}
\newcommand{\fMky}{\operatorname{\mathit{f}-Mky}}
\newcommand{\Mod}{\operatorname{Mod}}
\newcommand{\Mon}{\operatorname{Mon}}
\newcommand{\normN}{\operatorname{N}}
\newcommand{\NN}{\mathbb{N}}
\newcommand{\OO}{\operatorname{O}\nolimits}
\newcommand{\Rep}{\operatorname{Rep}}
\newcommand{\Res}{\operatorname{Res}}
\newcommand{\RT}{\operatorname{RT}}
\newcommand{\LieS}{\operatorname{S}\nolimits}
\newcommand{\catSS}{\mathbb{S}}
\newcommand{\Set}{\operatorname{Set}}
\newcommand{\SO}{\operatorname{SO}}
\newcommand{\Spn}{\operatorname{Spn}}
\newcommand{\StMon}{\operatorname{StMon}}
\newcommand{\TT}{\mathbb{T}}
\newcommand{\Tamb}{\operatorname{Tamb}}
\newcommand{\TambLS}{\Tamb_{\ell\mathrm{s}}}
\newcommand{\TambS}{\Tamb_{\mathrm{s}}}
\newcommand{\Tr}{\operatorname{Tr}}
\newcommand{\Tv}{\operatorname{Tv}}
\newcommand{\TV}{\operatorname{TV}}
\newcommand{\catV}{\mathcal{V}}
\newcommand{\VCat}{\operatorname{\mathcal{V}\mbox{-}Cat}}
\newcommand{\Vect}{\operatorname{Vect}}
\newcommand{\Weyl}{\operatorname{W}}
\newcommand{\XX}{\mathbb{X}}
\newcommand{\catX}{\mathcal{X}}
\newcommand{\YY}{\mathbb{Y}}
\newcommand{\catY}{\mathcal{Y}}
\newcommand{\YB}{\operatorname{YB}}
\newcommand{\ZZ}{\mathbb{Z}}
\newcommand{\catZ}{\mathcal{Z}}
\newcommand{\cntrZ}{\mathcal{Z}}
\newcommand{\add}{\operatorname{add}}
\newcommand{\elless}{\ell\mathrm{s}}
\newcommand{\ess}{\mathrm{s}}
\newcommand{\ev}{\operatorname{ev}}
\newcommand{\ffam}{\operatorname{fam}}
\newcommand{\hocolim}{\operatorname{hocolim}}
\newcommand{\inv}{\mathrm{inv}}
\newcommand{\Implies}{\Longrightarrow}
\newcommand{\mmod}{\operatorname{mod}}
\newcommand{\ob}{\operatorname{ob}}
\newcommand{\op}{\operatorname{op}}
\newcommand{\pr}{\operatorname{pr}^{}}
\newcommand{\rev}{\operatorname{rev}}
\newcommand{\val}{\operatorname{v}}
\newcommand{\zz}{\operatorname{z}}

\newcommand{\longto}{\longrightarrow}
\newcommand{\longfrom}{\longleftarrow}
\newcommand{\ldual}[1]{\mathord{{\let\nolimits\relax\sideset{^\wedge}{}{#1}}}}
\newcommand{\laction}[2]{\mathord{{\let\nolimits\relax\sideset{^{#1}}{}{#2}}}}
\newcommand{\conj}[2]{\mathord{{\let\nolimits\relax\sideset{^{#1}}{}{#2}}}}
\newcommand{\xylongto}[2][2pt]{{\UseTips{}\xymatrix @C+#1{ \ar[r]^{#2}&}}}

\begin{document}
\author{Ross Street\footnote{The author gratefully acknowledges the support of an Australian Research Council Discovery Grant DP1094883 and of the \'Ecole de Math\'ematique, Universit\'e catholique de Louvain.}}
 
\title{Monoidal categories in, and linking, geometry and algebra}
\date{}
\maketitle


\begin{abstract}
\noindent This is a report on aspects of the theory and use of monoidal categories.
The first section introduces the main concepts through the example of the category of vector spaces.
String notation is explained and shown to lead naturally to a link between knot theory and monoidal categories.
The second section reviews the light thrown on aspects of representation theory by the machinery of monoidal category theory, machinery such as braidings and convolution.
The category theory of Mackey functors is reviewed in the third section. Some recent material and a conjecture concerning monoidal centres is included. 
The fourth and final section looks at ways in which monoidal categories are, and might be, used for new invariants of low-dimensional manifolds and for the field theory of theoretical physics.

\end{abstract}

\section*{Introduction}

In essence, this paper consists of the notes of four lectures delivered in May 2011 as part of the Chaire de la Vall\'ee Poussin 2011\footnote{\url{http://www.uclouvain.be/15551.html}}. 
The third and fourth lectures were also part of the conference \emph{Category Theory, Algebra and Geometry} on 26 and 27 May 2011 in Louvain-la-Neuve, Belgium.    

The role of monoidal categories in mathematics can be expressed simplistically in an equation:
 
\[
\frac{\mathrm{Point \ \ in \ \ Euclidean \ \ space}}{\mathrm{Vector \ \ space}}=\frac{\mathrm{Vector \ \ space}}{\mathrm{Monoidal \ \ category}}\,.
\]

While the more precise subject of the paper can be gleaned from the table of contents, I would like to mention that the preparation of this material took on a life of its own. What forced itself on me was a strong feeling for the importance of the concept of duoidal category (called 2-monoidal category in~\cite{AguiarMahajan}) and for the interesting questions associated with the construction of the centre of a monoidal category. I hope the paper manages to capture and impart some of my enthusiasm.     

It is a pleasure to acknowledge the significant help from Cathy Brichard at Universit\'e catholique de Louvain who turned my handwritten lecture notes into \LaTeX . Congratulations to Tom Booker who in November 2011 facilitated my looking at and writing in \LaTeX\  code. 
Even so, I have much to learn, especially when it comes to diagrams: so my thanks go to Mark Weber and Craig Pastro for being willing helpers with that. 
In the final stages of preparation of this publication, I am grateful that Ross Moore came to the rescue for many vast improvements in the diagrams, a few now coming from our book~\cite{RossQG} for which he was the Technical Editor. 
Along with the standard references~\cite{CWM} and~\cite{Borceux}, 
our book provides a good supplement to the present Section~\ref{Lecture1}. 
Furthermore Section \ref{Lecture2} is dedicated to Brian Day and some of the material of Section~\ref{Lecture4} is joint with him.  

\tableofcontents

\clearpage
\section{From linear algebra to knot theory via categories\label{Lecture1}}

\subsection{Introduction to categories}

Formalizing properties of the Euclidean plane into the structure of vector space, 
we are able to transfer low-dimensional thinking to obtain precise results in higher dimensions.

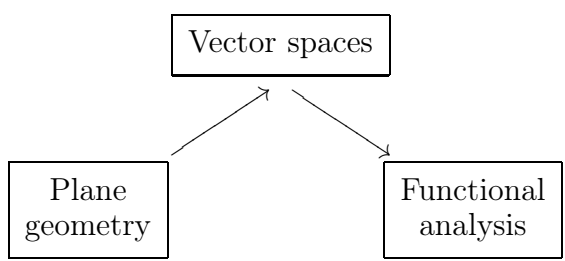
\[
\xygraph{ !{\UseTips}
 !{/r6pc/: /u4pc/::}
 []*++[F-]+!C\txt{Vector spaces} !{+D*+!C\txt{ \ }} 
 ( [dl] *++[F-]+\txt{Plane\\geometry} : ?
 , ? : [dr] *++[F-]+\txt{Functional\\analysis}
 )
}\]

\noindent Formalizing properties of the totality of vector spaces into the structure of monoidal category, we are able to transfer linear algebra into some perhaps surprising areas.

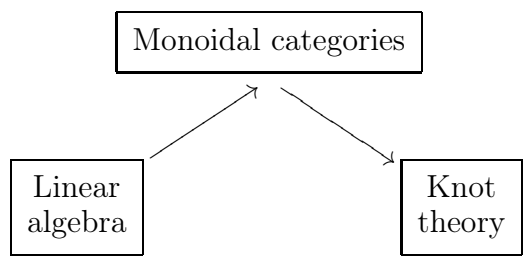
\[
\xygraph{  !{\UseTips}!{/u4pc/::}
 !{/r6pc/: /u4pc/::}
 []*++[F-]+!C\txt{Monoidal categories} !{+D*+!C\txt{ \ }} 
 ( [dl] *++[F-]+\txt{Linear\\algebra} : ?
 , ? : [dr] *++[F-]+\txt{Knot\\theory}
 )
}\]

Klein recognized the importance of structure-respecting transformations in geometry.
These transformations were taken to be invertible, forming a groupoid.
General structure-respecting functions are also crucial.
The notion of ``category'' captures that idea.

For example, the \emph{category} $\Vect $ of (say, complex) vector spaces consists not only of the vector spaces $V$ (\emph{objects}) but also the linear functions $f:V\to W$ between them (\emph{morphisms}).
The basic operation is composition
 $$
 \xymatrix{ & W \ar[dr]^g &
                \\ V \ar[ur]^f  \ar[rr]_{g\circ f} & & X
                }
 $$ 
of the morphisms.
Each object $V$ has an identity morphism $1_V : V \longto V$.

A first course on linear algebra introduces the operation of \emph{direct sum} $V \oplus W$ of vector spaces and the quantity \emph{dimension}. If $V$ and $W$ are finite dimensional,
\[
\dim (V \oplus W) = \dim V + \dim W.
\]
\emph{Bilinear functions} $h : V \times W \longto X$ would also be defined.
Yet the operation of \emph{tensor product}   $V \otimes W$ might await a course on multilinear algebra or module theory.
For $V$ and $W$ finite dimensional, 
\[
\dim (V \otimes W) = \dim V \times \dim W.
\]
Bilinear functions $h : V \times W \longto X$ are in canonical bijection with linear functions $f : V \otimes W \longto X$.
There is also the vector space $X^V$ of linear functions from $V$ to $X$.
Linear functions $f : V \otimes W  \longto X$ are in canonical bijection with linear functions
\[
g : W  \longto X^V.
\]

A \emph{category}  $\catC$ consists of a set of \emph{objects} and, for each pair of objects $A, B$, a set $\catC (A,B)$ of morphisms $f : A  \longto B$, together with an associative composition rule $\circ$ with identities $1_A : A  \longto A$.
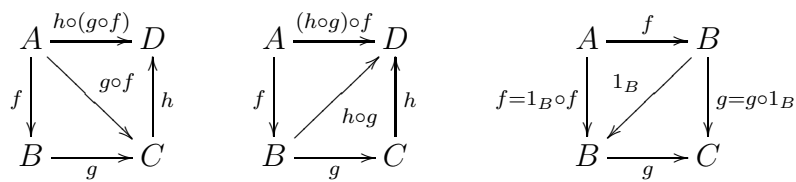
\[  \xymatrix@R+2mm @C+2mm{
    A \ar[r]^{h \circ (g\circ f)} \ar[d]_f \ar[rd]^{g \circ f} & D   \\
    B \ar[r]_g    & C \ar[u]_h
  } \qquad
  \xymatrix@R+2mm @C+2mm{
    A \ar[r]^{(h \circ g)\circ f} \ar[d]_f  & D   \\
    B \ar[r]_g  \ar[ru]_{h \circ g}  & C \ar[u]_h}
 \qquad
   \xymatrix@R+2mm @C+2mm{
    A \ar[r]^f\ar[d]_{f=1_B \circ f}   & B \ar[ld]_{1_B}  \ar[d]^{g=g\circ 1_B}\\
    B \ar[r]_g     & C }
\]

In a category, we can speak of \emph{commutative diagrams}:
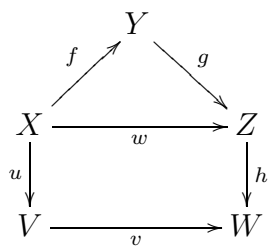
\begin{equation*}
\xymatrix{
& Y \ar[rd]^-{ g}  & \\
X \ar[ru]^-{f} \ar[d]_-{u} \ar[rr]_-{w} & & Z \ar[d]^-{h} \\
V \ar[rr]_-{v} & & W }
\end{equation*}
\begin{center}
$h \circ w = v \circ u$ \quad and \quad $g \circ f = w$, \quad
so $v \circ u = h \circ g \circ f$.
\end{center}

Categories are themselves mathematical structures: so we should look at morphisms between them.
A \emph{functor} $T : \catC  \longto \catH$ assigns an object $TA$ of $\catH$ to each object $A$ of $\catC$, 
a morphism $Tf : TA  \longto TB$ in $\catH$ to each $f : A  \longto B$ in $\catC$, such that
\[
T1_A = 1_{TA} \quad \mbox{and} \quad T(g \circ f) = Tg \circ Tf.
\]
For example, each vector space $V$ determines a functor
\[
T = V\otimes - : \Vect  \longto \Vect 
\]
defined by
\begin{align*}
TA &= V \otimes A  \\
Tf &= 1_V \otimes f : V \otimes A \longto V \otimes B\,,   \qquad  v \otimes a \longmapsto v \otimes f(a)\,. 
\end{align*}
In an obvious sense of \emph{product category}, tensor product is a functor
\[
- \otimes - : \Vect  \times \Vect  \longto \Vect\,.
\]

Categories exhibit a higher-dimensional structure in that there are morphisms between functors.
Suppose $S$ and $T : \catC \longto \catH$ are functors. A \emph{natural transformation}
\[
\theta : S \Longrightarrow T
\]
is a family of morphisms
\[
\theta_A : SA \longto TA
\]
in $\catH$, indexed by the objects $A$ of $\catC$, such that the square
\[
\xymatrix{
   S A \ar[r]^{\theta_A} \ar[d]_{Sf}   & TA  \ar[d]^{Tf}\\
    SB \ar[r]_{\theta_B}    & TB}
\]
commutes for all $f : A \longto B$ in $\catC$.

\medskip

For example, each linear function $t : V \longto W$ determines a natural transformation 
$t \otimes - : V \otimes - \longto W \otimes -$ with components
\begin{align*}
  t \otimes 1_A : &V \otimes A \longto W \otimes A \\
 & v \otimes a \longmapsto t(v) \otimes a\,. 
\end{align*}

\subsection{Introduction to monoidal categories}

We can now define a notion developed by Mac Lane~\cite{ML1963}, 
B\'enabou~\cite{Ben1963} and Eilenberg--Kelly~\cite{EilKel1966}. 
A \emph{monoidal category} consists of a category $\catC$ equipped with a functor
\[
- \otimes - : \catC \times \catC \longto \catC,
\]
an object $I$, and invertible natural families of morphisms
\begin{align*}
 & \alpha^{}_{A,B,C} : (A \otimes B) \otimes C \longto A \otimes (B \otimes C)\\
 & \lambda_A : I \otimes A \longto A\,, \qquad \rho_A : A \otimes I \longto A
\end{align*}
such that the following commute.

\[
\xymatrix{
& (A \otimes (B \otimes C)) \otimes D \ar[rd]^-{\phantom{A}\alpha^{}_{A,B\otimes C,D}}  & \\
((A \otimes B )\otimes C) \otimes D \ar[ru]^-{\alpha^{}_{A,B,C}\otimes 1_{D}\phantom{A}} \ar[d]_-{\alpha^{}_{A\otimes B,C,D}} & & A \otimes ((B \otimes C) \otimes D ) \ar[d]^-{1_{A}\otimes\alpha^{}_{B,C,D}} \\
(A \otimes B) \otimes (C \otimes D) \ar[rr]_-{\alpha^{}_{A,B,C\otimes D }} & & A \otimes (B \otimes (C \otimes D)) 
}\] 
\[
\xymatrix{(A \otimes I)\otimes B \ar[rr]^{\alpha^{}_{A,I,B}} \ar[rd]_{\rho_A \otimes 1_B} 
& & A \otimes (I \otimes B) \ar[ld]^{1_A \otimes \lambda_B} \\
& A \otimes B}
\]

\begin{Example} 
Take $\catC = \Vect $, usual tensor product of vector spaces as $\otimes$, and $I = \CC$.
\end{Example}

\subsection{String diagrams}

Penrose introduced \emph{string notation} for multilinear algebra.
This was adapted to monoidal categories by Joyal--Street~\cite{GTC}.
In any category $\catC$, we can write $A \stackrel{f}{\longto} B$ or as the following. 

 \[ 
 \xygraph{{f} *\xycircle<8pt>{-}="m" "m"(-[u] [l(.3)d(.25)] {A},-[d] [l(.3)u(.25)] {B})}
\] 
Both notations are ``1-dimensional''.

In a monoidal category $\catC$, we use two dimensions: composition is up-down while tensoring is left-right.
Instead of $A \otimes C \stackrel{f}{\longto} B \otimes C \otimes A$, we can depict as follows.   
\[ 
\xygraph{{f} *\xycircle<8pt>{-}="m" "m"(-[l(.5)u] [l(.1)d(.4)] {A},-[r(.5)u] [r(.1)d(.4)] {C},-[l(.75)d] [l(.1)u(.4)] {B},-[d] [l(.15)u(.3)] {C},-[r(.75)d] [r(.1)u(.4)] {A})} 
\]
A morphism $I \stackrel{f}{\longto}  A \otimes B$ is denoted as follows.
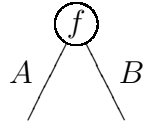
\[ 
\xygraph{!{0;(.5,0):(0,2)::}{f}*\xycircle<8pt>{-}="m" "m"(-[dl] [l(.15)u(.5)] {A},-[dr] [r(.15)u(.5)] {B})} 
\]
In this way, the monoidal structural items $\otimes, I, \alpha, \rho, \lambda$ do not appear.

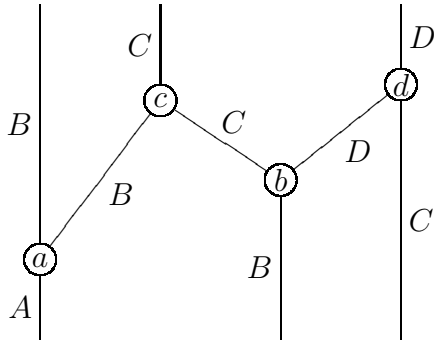
\[ 
\xygraph{ !{/r2.5pc/:}
{a} *\xycircle<6pt>{-}="a" [u(2)r(1.5)] {c} *\xycircle<6pt>{-}="c" [d(1)r(1.5)] {b} *\xycircle<6pt>{-}="b" [r(1.5)u(1.2)] {d} *\xycircle<6pt>{-}="d" "a" (-[u(3.2)]^-*{B},-[d]_-*{A},-"c"_-*{B}(-[u(1.2)]^-*{C},-"b"^-*{C}(-[d(2)]_-*{B},-"d"_-*{D}(-[u]_-*{D},-[d(3.2)]^-*{C}))))} 
\]
A diagram $\Gamma$ as above, labelled in $\catC$, is called a \emph{progressive plane string diagram}; 
see Joyal--Street~\cite{GTC}.  
It means that we have morphisms
\[
B \otimes B \stackrel{a}{\longto} A\,, \quad
C \otimes D \stackrel{b}{\longto} B\,,\quad
C   \stackrel{c}{\longto} B \otimes C\,,\quad
D \stackrel{d}{\longto}  D \otimes C\,.  
\]
The \emph{value} of the diagram is the composite
\begin{align*}
 \val(\Gamma) =& (B \otimes C \otimes D \xylongto[16pt]{1_B\otimes c \otimes d} 
 B \otimes B \otimes C \otimes D \otimes C
\\
 & \xylongto[20pt]{1\otimes 1\otimes b \otimes 1} 
 B \otimes B \otimes B \otimes C   \xylongto[10pt]{a \otimes 1\otimes  1} 
 A \otimes B  \otimes C)\,.
\end{align*}
It is obtained by composing the more elementary values of horizontal layers. 
This value is independent of \emph{deformation} in a natural sense. 
For example, the diagram  

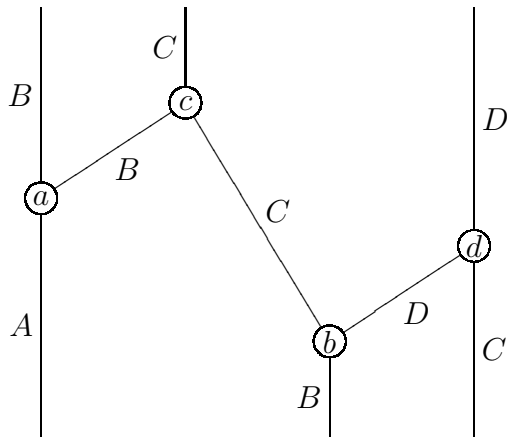
\[ 
\xygraph{{a} *\xycircle<6pt>{-}="a" [r(1.5)u] {c} *\xycircle<6pt>{-}="c" [r(1.5)d(2.5)] {b} *\xycircle<6pt>{-}="b" [r(1.5)u] {d} *\xycircle<6pt>{-}="d" "a" (-[u(2)]^-*{B},-[d(2.5)]_-*{A},-"c"_-*{B}(-[u]^-*{C},-"b"^-*{C}(-[d]_-*{B},-"d"_-*{D}(-[u(2.5)]_-*{D},-[d(2)]^-*{C}))))} 
\]
is a deformation of $\Gamma$ and its value is
\begin{align*}
\val(\Gamma) =& (B \otimes C \otimes D \xylongto[14pt]{1_B\otimes c \otimes 1} 
 B \otimes B \otimes C \otimes D \xylongto[10pt]{a \otimes 1 \otimes 1} 
 A \otimes C \otimes D
\\
 & \xylongto[10pt]{1\otimes 1\otimes d} 
  A \otimes C \otimes D \otimes C   \xylongto[10pt]{1 \otimes b \otimes  1} 
  A \otimes B  \otimes C)\,.
\end{align*}

\goodbreak\noindent
The notation handles units well:
if $I   \stackrel{a}{\longto}   A \otimes B$ and $C \stackrel{b}{\longto}  I$, 
then the following three string diagrams all have the same value.
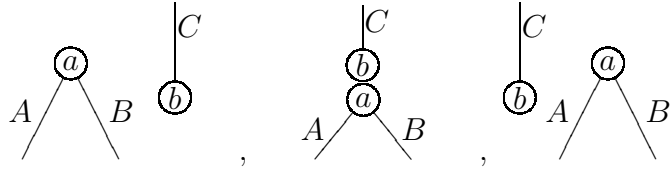
\[ 
\xygraph{!{0;(1.25,0):(0,.65)::} {\xybox{\xygraph{{\xybox{\xygraph{!{0;(.5,0):(0,2)::} {a} *\xycircle<6pt>{-}="m" "m"(-[dl] [l(.05)u(.5)] {A},-[dr] [r(.05)u(.5)] {B})}}} [u(.5)r(1.5)] {\xybox{\xygraph{{b} *\xycircle<6pt>{-}="m" "m" -[u] [r(.15)d(.25)] {C}}}}}}}
[dr] {,} [ur]
{\xybox{\xygraph{{\xybox{\xygraph{!{0;(.5,0):(0,1.24)::} {a} *\xycircle<6pt>{-}="m" "m"(-[dl] [l(.05)u(.5)] {A},-[dr] [r(.05)u(.5)] {B})}}} [u(.61)r(.41)] {\xybox{\xygraph{!{0;(1,0):(0,.62)::} {b} *\xycircle<6pt>{-}="m" "m" -[u] [r(.15)d(.25)] {C}}}}}}}
[dr] {,} [ur]
{\xybox{\xygraph{{\xybox{\xygraph{!{0;(.5,0):(0,2)::} {a} *\xycircle<6pt>{-}="m" "m"(-[dl] [l(.05)u(.5)] {A},-[dr] [r(.05)u(.5)] {B})}}} [u(.5)l(.5)] {\xybox{\xygraph{{b} *\xycircle<6pt>{-}="m" "m" -[u] [r(.15)d(.25)] {C}}}}}}}} 
\]

\subsection{Duals}

We can look at some concepts which make sense in any monoidal category.

\medskip\noindent
A \emph{duality} between objects $A$ and $B$ is a pair of morphisms
\[
e : A \otimes B \longto I  \quad \mbox{and} \quad d : I \longto B \otimes A
\]
called the \textit{counit} and the \textit{unit}, respectively, satisfying

\[
 \xymatrix @R+3mm {
 A \otimes I \ar[r]^-{1\otimes d} \ar[rd]_{\rho_A}^\cong 
 & A \otimes (B\otimes A) {\alpha^{-1} \atop {\mbox{\large $\cong$}}}  (A\otimes B)\otimes A \ar[r]^-{e\otimes 1}
 & I \otimes A \ar[ld]^{\lambda_B}_\cong \\
 & B }
\]
\[
 \xymatrix@R+3mm {
 I \otimes B \ar[r]^-{d\otimes 1} \ar[rd]_{\lambda_B}^\cong 
 & (B\otimes A) \otimes B {\alpha \atop {\mbox{\large $\cong$}}}  B\otimes (A\otimes B)\ar[r]^-{1\otimes e}
 & B \otimes I \ar[ld]^{\rho_B}_\cong \\
 & B }
\]
In string notation, there are equalities as follows.
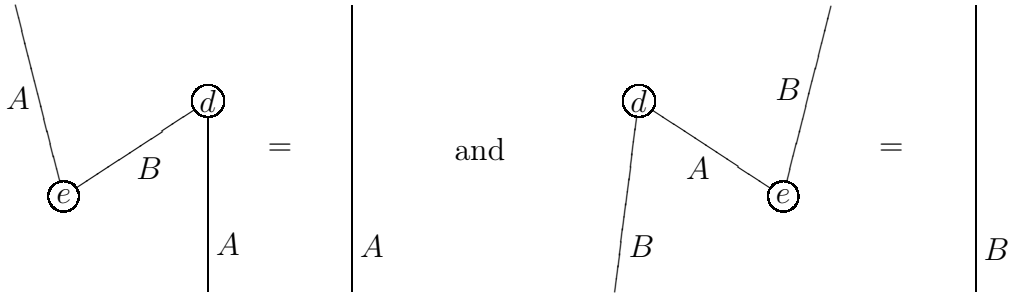
\[ 
\xygraph{!{0;(2.6,0):(0,.5)::} {\xybox{\xygraph{{e} *\xycircle{-}="e" [r(1.5)u] {d} *\xycircle{-}="a" "e"(-[u(2)l(.5)]^-*{A},-"a"_-*{B}-[d(2)]^(.75)*{A} [r(1.5)] -[u(3)]_(.16)*{A}) "a" [r(.75)d(.5)] {=}}}}  
 [r(1.2)] *\txt{and} [r(1.2)]
{\xybox{\xygraph{{d} *\xycircle{-}="a" [r(1.5)d] {e} *\xycircle{-}="e" "a"(-[d(2)l(.25)]^(.75)*{B},-"e"_-*{A}-[u(2)r(.5)]^-*{B} [r(1.5)] -[d(3)]^(.85)*{B}) "e" [r(1.125)u(.5)] {=}}}}} 
\]
Write $A \dashv B$  when such $e$ and $d$ exist: $A$ is a \emph{left dual} for $B$ while $B$ is a \emph{right dual} for $A$. 
Any two right duals are isomorphic. Write $A^\ast$ for a chosen right dual for~$A$. 
Call $\catC$ right \emph{autonomous} (``compact'', ``rigid'') when each object has a right dual. 
Choice gives a functor
\[
( \ )^\ast : \catC^{\op}   \longto \catC
\]
where ``$\op$'' means $( \ )^\ast$ is ``contravariant'': 
$A \stackrel{f}{\longto} C$ goes to  $A^\ast \stackrel{\ f^\ast}{\longleftarrow} C^\ast $.

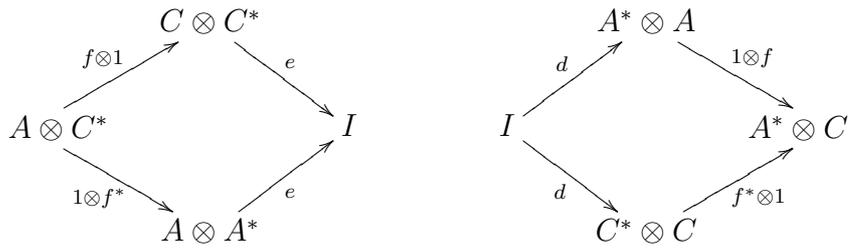
\[
\xymatrix{  & C\otimes C^\ast \ar[rd]^{e} \\
*+!L(.5){A \otimes C^\ast}\ar[ru]^{f\otimes 1} \ar[rd]_{1\otimes f^\ast} & & I \\
 & A\otimes A^\ast   \ar[ru]_{e} }  
\qquad\qquad
\xymatrix{  & A^\ast \otimes A \ar[rd]^-{1\otimes f} \\
I \ar[ru]^-{d} \ar[rd]_-{d} & & *+!R(.5){A^\ast \otimes C} \\
 & C^\ast \otimes C   \ar[ru]_-{f^\ast \otimes 1} }
 \]

\medskip\noindent
We suppress $e$ and $d$ from notation by writing them as
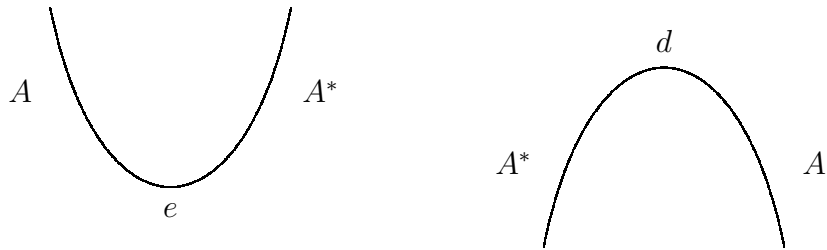
\[ \kern-1cm
\xygraph{!{0;(-2.5,0):(0,1)::} [dd]*{}="l" [r] *{}="r" 
 "l"-@`{"l"+(.2,1),"r"+(-.2,1)}"r"|*++!U{e} 
"l" [l(.125)u(.35)] {A^*} "r" [r(.125)u(.35)] {A} 
}
\kern-3cm
\xygraph{!{0;(2.5,0):(0,1)::} [d] *{}="l" [r] *{}="r"
 "l"-@`{"l"+(.2,1),"r"+(-.2,1)}"r" |*++!D{d}
"l" [l(.125)u(.35)] {A^*} "r" [r(.125)u(.35)] {A}  
} 
\]

\noindent 
The duality conditions become string straightening rules.
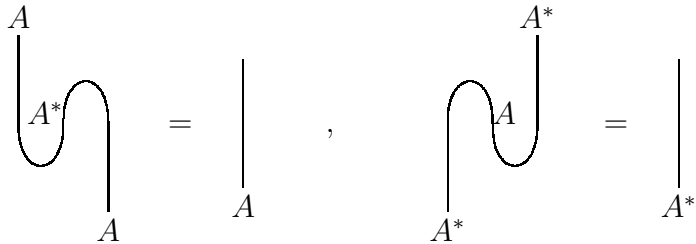
\[
\xygraph{ !{/r1.4pc/: /u2.75pc/::}
 []*+!D{A}
!{\xcapv-@(0)}
!{\vcap[-]}
[r] (?*!DR{A^*} , !{\vcap} )
[r] !{\xcapv-@(0)} *+!U{A}
} \quad = 
\xygraph{
!{/r1.3pc/: /u2pc/::}
!{\xcapv-@(0)}
!{\xcapv-@(0)} *+!U{A}
} \qquad,\qquad
\xygraph{ !{/r1.4pc/: /u2.75pc/::}
 [rr]*+!D{A^*}
  !{\xcapv-@(0)}
 [l] !{\vcap[-]}
 (?*!DL{A} ,  [l] !{\vcap} )
 [l] !{\xcapv-@(0)} *+!U{A^*}
} \quad = 
\xygraph{
!{/r1.3pc/: /u2pc/::}
!{\xcapv-@(0)}
!{\xcapv-@(0)} *+!U{A^*}
} \]
This allows anticlockwise looping by taking $( \ )^\ast$.
$f^\ast$ is depicted as follows. 
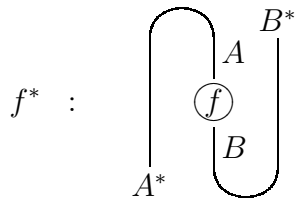
\[ 
f^* \;\; : 
\xygraph{ !{/r2pc/: /u2pc/::}
 [rr]*+!D{B^*}
  !{\xcapv-@(0)} !{\xcapv-@(0)}
 [l] !{\vcap[-]} - [u]*+<8pt>[Fo]{f} _(.3)*{B}
 ( ?-[u]_(.8)*{A} [l] !{\vcap} )
 [lu] !{\xcapv-@(0)\xcapv-@(0)} *+!U{A^*}
} \]
Diagrams are not necessarily ``progressive''  any more.   They can backtrack.

\begin{Example}
An object of $\Vect $ has a right (left) dual if and only if it is finite dimensional (exercise!). 
If $V$ has basis $v_1,\ldots,v_n$ and $v^\ast_1,\ldots,v^\ast_n$ is the ``dual basis'' of $V^\ast = \CC^V$ then we have
\begin{align*}
 V \otimes V^\ast &\stackrel{e}{\longto} \CC    
   & \CC  &\stackrel{d}{\longto}   V^\ast \otimes V  \\
  v \otimes \varphi &\longmapsto \varphi (v)   
   & 1 &\longmapsto \sum^n_{i=1} v^\ast_i \otimes v_i\,.
\end{align*}
\end{Example}

\subsection{Braidings}

A \textit{braiding}~\cite{BTC} on a monoidal category $\catC$ is a natural family of invertible morphisms 
\[
\gamma^{}_{A,B} : A \otimes B  \xylongto{\cong} 
 B \otimes A
\]
satisfying (ignoring $\alpha, \lambda, \rho$)
\[
 \xymatrix @C-10mm{
A \otimes B \otimes C \ar[rr]^{\gamma_{A\otimes B,C}^{}} \ar[rd]_-{1\otimes \gamma_{B,C}^{}} 
& & C \otimes A \otimes B  \\
& A \otimes C \otimes B\ar[ru]_-{\gamma_{A,C}^{} \otimes 1}}  
 \qquad 
\xymatrix @C-10mm{
A \otimes B \otimes C \ar[rr]^{\gamma_{A, B\otimes C}^{}} \ar[rd]_-{\gamma_{A,B}^{} \otimes 1} 
& & B \otimes C \otimes A  \\
& B \otimes A \otimes C \ar[ru]_-{1\otimes \gamma_{A,C}^{}}}   
\]

\begin{Example}
In $\Vect$, $\gamma^{}_{A,B} (a\otimes b) = b \otimes a$.
\end{Example}

\noindent
By writing $\gamma$ as a \emph{crossing}
\[
\xygraph{
 []{\gamma^{}_{A,B} :} [r][u(0.5)] *+!DR{A}( !{\xunderv}, [r]*+!DL{B} , [dr]*+!UL{A} ,  [d]*+!UR{B} )}
\]
our string diagrams move into three-dimensional Euclidean space.   The axioms say
\[
\xygraph{[u(.75)]*+{A} ( [r]*+{B} , [d(1.5)l(.5)]*+{B'} , [dr(1.5)]*+{A'},
  [d]*++=[o][Fo]{b}="b"-"B'" , [dr]*++=[o][Fo]{a}="a"-"A'" ,
!{\xunderv ~{"A"}{"B"}{"b"}{"a"}} 
  )}  = 
\xygraph{[u(.75)]*+{A} ( [rr]*+{B} , [d(1.5)r(.5)]*+{B'} , [dr(1.5)]*+{A'},
  [dr(.5)r]*++=[o][Fo]{b}="b"-"B" , [dr(.5)]*++=[o][Fo]{a}="a"-"A" ,
!{\xunderv ~{"a"}{"b"}{"B'"}{"A'"}} 
  )}  
\]
\[
\xygraph{ !{\knotholesize{15pt}}
!{/r1.8pc/: /u4pc/::}
[d(.5)] ( ?*+!DR{A} , [r] *+!DL{C}
 , !{\xunderv @(.45)}
, [r(.3)] ( ?*+!D{B} , ?- [rd]
))
}
=
\xygraph{
!{/r1.3pc/: /u2pc/::}
 !{\xcapv-@(0)}
!{\vtwistneg} ( []*+!U{C}  [r]*+!U{A}  [r]*+!U{B} ,
[uur] !{\vtwistneg}
[r]!{\xcapv@(0)}
)
}
\quad,\qquad
\xygraph{ 
!{/r1.8pc/: /u4pc/::}
[d(.5)] ( ?*+!DR{A} , [r(.7)] *+!D{B}  , [r] *+!DL{C} ,
!{\xunderv}
, [r(.7)] - [ld] |(.35)*{\khole}
}
= 
\xygraph{
!{/r1.3pc/: /u2pc/::}
[r] !{\vtwistneg}
!{\xcapv-@(0)}
[uurr]
!{\xcapv@(0)}
[l] !{\vtwistneg} 
[l]*+!U{B}  [r]*+!U{C}  [r]*+!U{A} 
}\,.
\]

A monoidal category with chosen braiding is called \emph{braided} in the terminology of Joyal--Street~\cite{BTC}. 
The braiding automatically satisfies what is called the \textit{Yang--Baxter equation} 
or the \textit{braid identity} or a \textit{Reidemeister move} in different contexts. 
Here is the proof:

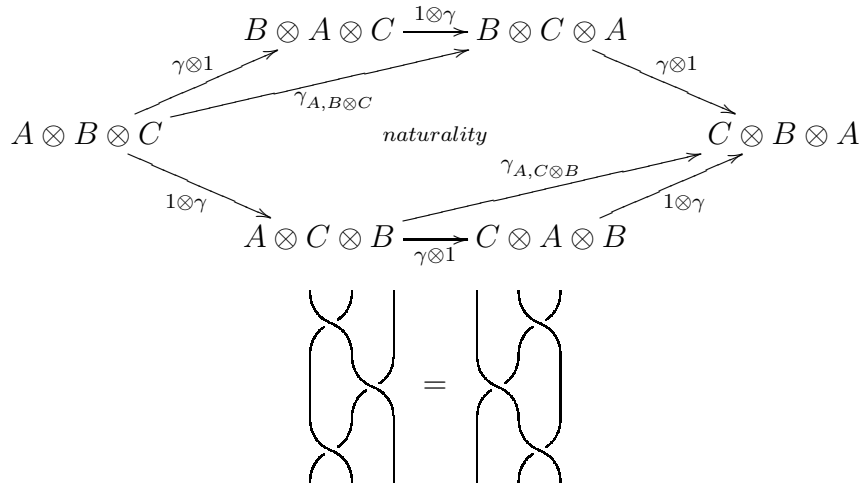
\[
\xymatrix{
  & B\otimes A \otimes C \ar[r]^-{1\otimes \gamma} 
    & B\otimes C \otimes A \ar[rd]^-{\gamma \otimes 1} \\
 A\otimes B\otimes C 
   \ar[ru]^-{\gamma \otimes 1} \ar[rru]_-{\;\gamma^{}_{A,B\otimes C}} 
    \ar[rd]_-{1\otimes \gamma} \ar@{}[rrr]|{naturality} 
    & & & C\otimes B \otimes A \\
 & A\otimes C \otimes B \ar[rru]^-{\gamma^{}_{A,C\otimes B}} \ar[r]_-{\gamma \otimes 1} 
   & C\otimes A \otimes B \ar[ru]_-{1\otimes \gamma} }
\]
\kern-20pt
\[
\xygraph{
!{/r1.3pc/: /u2pc/::}
!{\vtwistneg}
!{\xcapv-@(0)}
!{\vtwistneg}
[uuurr]
!{\xcapv@(0)}
[l]!{\vtwistneg}
[r]!{\xcapv@(0)}
[r][u(1.5)]
{=}
[r][u(1.5)]
!{\xcapv-@(0)}
!{\vtwistneg}
!{\xcapv-@(0)}
[uuur]
!{\vtwistneg}
[r]!{\xcapv@(0)}
[l]!{\vtwistneg}
}
\]
\noindent
This leads naturally to our second example of braided monoidal category.

\subsubsection*{{\it Braid category $\BB$}}

Objects are natural numbers 0, 1, 2, \ldots
\[
\BB(m,n) = \begin{cases}
\emptyset  &  \text{when } m \neq n  \\
\mbox{braid group} \ B_n \text{ on } n \text{ strings } & \text{when } m=n
\end{cases}
\]
The functor $\BB \times \BB \stackrel{\otimes}{\longto} \BB$ is defined on objects by 
$(m,n) \longmapsto m+n$ and on morphisms as indicated by:
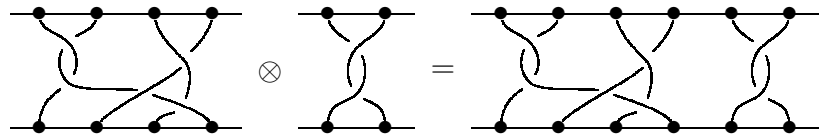
\[
\xy
0;/r.18pc/:
(0,0)*{\xy
(-15,20);(25,20) **\dir{-};
(-15,0);(25,0) **\dir{-};
(-10,20)*{\bullet}="T1"; 
(0,20)*{\bullet}="T2"; 
(10,20)*{\bullet}="T3"; 
(20,20)*{\bullet}="T4";
(-10,0)*{\bullet}="B1"; 
(0,0)*{\bullet}="B2"; 
(10,0)*{\bullet}="B3"; 
(20,0)*{\bullet}="B4";
(-5,15)*{\hole}="a";
(-5,15)*{}="a'";
(-5,8)*{\hole}="b";
(-5,8)*{}="b'";
"T1";"a'" **\crv{"T1"+(0,-3)&"a'"+(-2,2)};
"T2";"a" **\crv{"T2"+(0,-3)&"a"+(2,2)};
"a'";"b" **\crv{"a'"+(2,-2)&"b"+(2,2)};
"a";"b'" **\crv{"a"+(-2,-2)&"b'"+(-2,2)};
"b";"B1" **\crv{"b"+(-2,-2)&"B1"+(0,5)};
(15,12)*{\hole}="c";
(15,12)*{}="c'";
"T3";"c'" **\crv{"T3"+(0,-3)&"c'"+(-2,2)};
"T4";"c" **\crv{"T4"+(0,-3)&"c"+(2,2)};
"c";"B2" **\crv{"c"+(2,-2)&"B2"+(0,3)} \POS?(.5)*{\hole}="d";
"b'";"d" **\crv{"b'"+(2,-2)&"d"+(-2,1)};
"d";"B4" **\crv{"d"+(2,-1)&"B4"+(0,3)} \POS?(.5)*{\hole}="e";
"c'";"e" **\crv{"c'"+(2,-2)&"e"+(1,2)};
"e";"B3" **\crv{"e"+(-1,-2)&"B3"+(0,3)};
\endxy};
(25,0)*{\otimes};
(40,0)*{\xy
(-15,20);(5,20) **\dir{-};
(-15,0);(5,0) **\dir{-};
(-10,20)*{\bullet}="T1"; 
(0,20)*{\bullet}="T2"; 
(-10,0)*{\bullet}="B1"; 
(0,0)*{\bullet}="B2"; 
(-5,14)*{\hole}="a";
(-5,14)*{}="a'";
(-5,6)*{\hole}="b";
(-5,6)*{}="b'";
"T1";"a" **\crv{"T1"+(0,-3)&"a"+(-2,2)};
"T2";"a'" **\crv{"T2"+(0,-3)&"a'"+(2,2)};
"a";"b'" **\crv{"a"+(2,-2)&"b'"+(2,2)};
"a'";"b" **\crv{"a'"+(-2,-2)&"b"+(-2,2)};
"b'";"B1" **\crv{"b'"+(-2,-2)&"B1"+(0,5)};
"b";"B2" **\crv{"b"+(2,-2)&"B2"+(0,5)};
\endxy};
(55,0)*{=};
(90,0)*{\xy
(-15,20);(45,20) **\dir{-};
(-15,0);(45,0) **\dir{-};
(-10,20)*{\bullet}="T1"; 
(0,20)*{\bullet}="T2"; 
(10,20)*{\bullet}="T3"; 
(20,20)*{\bullet}="T4";
(30,20)*{\bullet}="T5";
(40,20)*{\bullet}="T6";
(-10,0)*{\bullet}="B1"; 
(0,0)*{\bullet}="B2"; 
(10,0)*{\bullet}="B3"; 
(20,0)*{\bullet}="B4";
(30,0)*{\bullet}="B5";
(40,0)*{\bullet}="B6";
(-5,15)*{\hole}="a";
(-5,15)*{}="a'";
(-5,8)*{\hole}="b";
(-5,8)*{}="b'";
"T1";"a'" **\crv{"T1"+(0,-3)&"a'"+(-2,2)};
"T2";"a" **\crv{"T2"+(0,-3)&"a"+(2,2)};
"a'";"b" **\crv{"a'"+(2,-2)&"b"+(2,2)};
"a";"b'" **\crv{"a"+(-2,-2)&"b'"+(-2,2)};
"b";"B1" **\crv{"b"+(-2,-2)&"B1"+(0,5)};
(15,12)*{\hole}="c";
(15,12)*{}="c'";
"T3";"c'" **\crv{"T3"+(0,-3)&"c'"+(-2,2)};
"T4";"c" **\crv{"T4"+(0,-3)&"c"+(2,2)};
"c";"B2" **\crv{"c"+(2,-2)&"B2"+(0,3)} \POS?(.5)*{\hole}="d";
"b'";"d" **\crv{"b'"+(2,-2)&"d"+(-2,1)};
"d";"B4" **\crv{"d"+(2,-1)&"B4"+(0,3)} \POS?(.5)*{\hole}="e";
"c'";"e" **\crv{"c'"+(2,-2)&"e"+(1,2)};
"e";"B3" **\crv{"e"+(-1,-2)&"B3"+(0,3)};
(35,14)*{\hole}="g";
(35,14)*{}="g'";
(35,6)*{\hole}="h";
(35,6)*{}="h'";
"T5";"g" **\crv{"T5"+(0,-3)&"g"+(-2,2)};
"T6";"g'" **\crv{"T6"+(0,-3)&"g'"+(2,2)};
"g";"h'" **\crv{"g"+(2,-2)&"h'"+(2,2)};
"g'";"h" **\crv{"g'"+(-2,-2)&"h"+(-2,2)};
"h'";"B5" **\crv{"h'"+(-2,-2)&"B5"+(0,5)};
"h";"B6" **\crv{"h"+(2,-2)&"B6"+(0,5)};
\endxy};
\endxy
\]
The isomorphisms $\alpha, \lambda, \rho$ are identities.

\medskip\noindent
The braiding $\gamma^{}_{m,n} $ is indicated by:
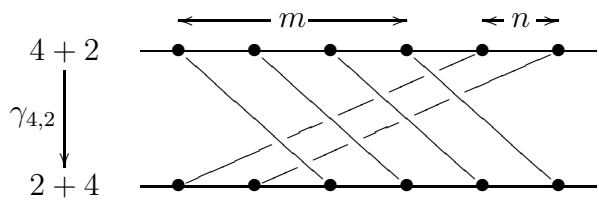
\[
\xy
(-25,18)*+{4+2};(-25,0)*+{2+4} **\dir{-} \POS?(.5)*+!R{\gamma_{4,2}} ?>*\dir{>};
(5,22)*+{m};(-10,22)*{} **\dir{-} ?>*\dir{>};
(5,22)*+{\hole};(20,22)*{} **\dir{-} ?>*\dir{>};
(35,22)*+{n};(30,22)*{} **\dir{-} ?>*\dir{>};
(35,22)*+{\hole};(40,22)*{} **\dir{-} ?>*\dir{>};
(-15,18);(45,18) **\dir{-};
(-15,0);(45,0) **\dir{-};
(-10,18)*{\bullet}="T1"; 
(0,18)*{\bullet}="T2"; 
(10,18)*{\bullet}="T3"; 
(20,18)*{\bullet}="T4";
(30,18)*{\bullet}="T5";
(40,18)*{\bullet}="T6";
(-10,0)*{\bullet}="B1"; 
(0,0)*{\bullet}="B2"; 
(10,0)*{\bullet}="B3"; 
(20,0)*{\bullet}="B4";
(30,0)*{\bullet}="B5";
(40,0)*{\bullet}="B6";
"T1";"B3" **\crv{} \POS?(.8333)*{\hole}="h" \POS?(.8)*{\hole}="d";
"T2";"B4" **\crv{} \POS?(.6667)*{\hole}="g" \POS?(.75)*{\hole}="c";
"T3";"B5" **\crv{} \POS?(.5)*{\hole}="f" \POS?(.6667)*{\hole}="b";
"T4";"B6" **\crv{} \POS?(.3333)*{\hole}="e" \POS?(.5)*{\hole}="a" ;
"T5";"a" **\crv{};
"a";"b" **\crv{};
"b";"c" **\crv{};
"c";"d" **\crv{};
"d";"B1" **\crv{};
"T6";"e" **\crv{};
"e";"f" **\crv{};
"f";"g" **\crv{};
"g";"h" **\crv{};
"h";"B2" **\crv{};
\endxy
\]

\begin{Proposition} 
$\BB$ is the free braided monoidal category generated by a single object.
\end{Proposition}

\noindent 
The single generating object is $1$ and tensoring the braiding $\gamma^{}_{1,1}$ on both sides with identity morphisms gives the generators $s_i$ for the braid group $B_n$ as indicated:
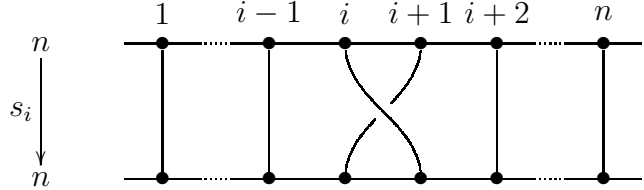
\[
\xy
(-30,18)*+{n};(-30,0)*+{n} **\dir{-} \POS?(.5)*+!R{s_i} ?>*\dir{>};
(-19,18);(-9,18) **\dir{-};
(-9,18);(-5,18) **\dir{.};
(-5,18);(35,18) **\dir{-};
(35,18);(39,18) **\dir{.};
(39,18);(49,18) **\dir{-};
(-19,0);(-9,0) **\dir{-};
(-9,0);(-5,0) **\dir{.};
(-5,0);(35,0) **\dir{-};
(35,0);(39,0) **\dir{.};
(39,0);(49,0) **\dir{-};
(-14,18)*{\bullet}="T1"+(0,4)*{1};
(0,18)*{\bullet}="T2"+(0,4)*{i-1};
(10,18)*{\bullet}="T3"+(0,4)*{i};
(20,18)*{\bullet}="T4"+(0,4)*{i+1};
(30,18)*{\bullet}="T5"+(0,4)*{i+2};
(44,18)*{\bullet}="T6"+(0,4)*{n};
(-14,0)*{\bullet}="B1"; 
(0,0)*{\bullet}="B2"; 
(10,0)*{\bullet}="B3"; 
(20,0)*{\bullet}="B4";
(30,0)*{\bullet}="B5";
(44,0)*{\bullet}="B6";
"T1";"B1" **\dir{-};
"T2";"B2" **\dir{-};
(15,9)*{\hole}="a";
(15,9)*{}="a'";
"T3";"a'" **\crv{"T3"+(0,-4)&"a'"+(-2,2)};
"T4";"a" **\crv{"T4"+(0,-4)&"a"+(2,2)};
"a";"B3" **\crv{"a"+(-2,-2)&"B3"+(0,4)};
"a'";"B4" **\crv{"a'"+(2,-2)&"B4"+(0,4)};
"T5";"B5" **\dir{-};
"T6";"B6" **\dir{-};
\endxy
\]
The Proposition means that there is an equivalence of categories
\[
\BStMon(\BB,\catC) \simeq \catC
\]
between the category of braiding-and-tensor-preserving functors $\BB \to \catC$ and the category $\catC$ itself, for every braided monoidal category $\catC$.

\bigskip

A \emph{Yang--Baxter operator} $y : A \otimes A \to A \otimes A$ on an object  $A$ of any monoidal category $\catC$ is an invertible morphism  satisfying commutativity of the following hexagon. There is an obvious category $\mathrm{YB}\catC$ of such pairs $(A,y)$.

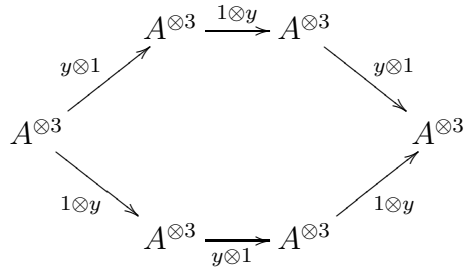
\begin{equation*}
\xymatrix{& A^{\otimes 3} \ar[r]^{1\otimes y} & A^{\otimes 3} \ar[rd]^{y\otimes 1} \\
 A^{\otimes 3}\ar[ru]^{y\otimes 1}  \ar[rd]_{1\otimes y}  & & & A^{\otimes 3} \\
 & A^{\otimes 3}  \ar[r]_{y\otimes 1 \ } & A^{\otimes 3} \ar[ru]_{1\otimes y} }
 \end{equation*}
If $\catC$ is braided, each object $A$ has a Yang--Baxter operator on it, namely,
\[
y =  \gamma^{}_{A,A} : A \otimes A \longto A \otimes A\,.
\]

\begin{Proposition} 
$\BB$  is the free monoidal category generated by an object bearing a Yang--Baxter operator.   That is, for every monoidal category $\catC$, 
\[
\StMon(\BB,\catC) \simeq \YB \catC\,.
\]
\end{Proposition} 

\begin{Proposition} 
Any right autonomous braided monoidal category is left autonomous.
\end{Proposition} 

\noindent
To see this we take $A^\ast \otimes A \stackrel{\gamma}{\longto}  A \otimes A^\ast   \stackrel{e}{\longto}  I$ 
and $I   \stackrel{d}{\longto}  A^\ast \otimes A \stackrel{\gamma^{-1}}{\longto}  A \otimes A^\ast$ as counit 
and unit for $A^\ast$ as left dual to $A$. 

\subsection{Trace, tangles and link invariants}

The \emph{trace} of an endomorphism
$ f : A    \longto   A $
in an autonomous braided monoidal category is defined by
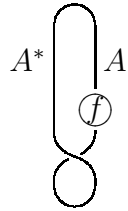
\[
\Tr (f) = \Bigl(I \stackrel{d}{\longto} A^\ast \otimes A \xylongto{1 \otimes f} 
  A^\ast \otimes A \stackrel{e'}{\longto}  I\Bigr)\,.
\]
\[
\xygraph{
!{/r1.3pc/: /u1.6pc/::}
!{\vcap} ( [r] - [dd] ^(.3)*{A}|(.8)*+[Fo]{f}
, ? - [dd] _(.3)*{A^*}
!{\vtwistneg}
!{\vcap[-]}
)}
\]

\begin{Example} 
If $V \in \Vect $ is finite dimensional with basis $v^{}_1,\ldots,v^{}_n$ and
$f$ is defined by $f(v^{}_i) = \sum_j a^{}_{ij} v^{}_j$, then
\begin{align*}
\Tr(f) :  \CC  &\longto \CC \\
                  1 &\longmapsto \displaystyle \sum_i a^{}_{ii}\,.
\end{align*}
\end{Example}

\subsubsection*{{\it Tangle category $\TT$}}
This is a monoidal category first defined by Yetter~\cite{Yetter1988}. 
The objects are words $- + + - - +$ in symbols $+$ and $-$.    
Morphisms are tangles
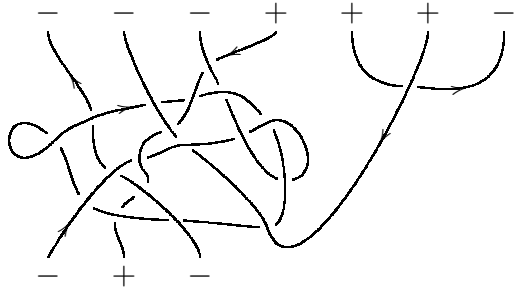
\[
\xy
(0,20)*+!D{-}="t1";(10,20)*+!D{-}="t2";(20,20)*+!D{-}="t3";
(30,20)*+!D{+}="t4";(40,20)*+!D{+}="t5"; (50,20)*+!D{+}="t6";(60,20)*+!D{-}="t7";
(0,-10)*+!U{-}="b1";(10,-10)*+!U{+}="b2";(20,-10)*+!U{-}="b3";
"t3"+(3,-8)="a";
"t2"+(8,-15)="c";
"t3";"a"*{\hole} **\crv{"t3"+(0,-3)&"a"+(-1,2)} \POS?(.5)="b";
"t4";"b"*{\hole} **\crv{"t4"+(0,-1)&"b"+(1,1)} \POS?(.7)*\dir{>};
"t2";"c"*{\hole} **\crv{"t2"+(0,-2)&"c"+(-3,2)} \POS?(.8)="d" \POS?(.6)="e";
"b"*{\hole};"d"*{\hole} **\crv{"b"-(2,2)&"d"+(3,1)} \POS?(.5)="f";
"e"*{\hole};"f"*{\hole} **\crv{"f"};
"f"*{\hole};"a" **\crv{"f"&"a"+(-2,0)};
"e"+(-8,-2)="g";
"e"*{\hole};"g" **\crv{"e"+(1,0)&"g"+(2,1)} \POS?(.5)*\dir{<};
"t1";"g"*{\hole} **\crv{"t1"+(0,-3)&"g"+(0,3)} \POS?(.5)*\dir{<};
"d"+(-4,-4)="h";
"h"*{\hole};"b1" **\crv{"h"+(-4,0)&"b1"+(1,2)} \POS?(.5)="j" \POS?(.25)="k" \POS?(.65)*\dir{<};
"h"*{\hole};"c" **\crv{"h"+(2,0)&"c"+(-2,0)};
"d"*{\hole};"h" **\crv{"d"-(3,1)&"h"+(0,2)};
"g";"g"+(-5,-3)="i" **\crv{"g"+(-2,-1)&"i"+(2,2)};
"i";"i"*{\hole} **\crv{"i"-(8,8)&"i"+(-8,8)};
"i"*{\hole};"j"*{\hole} **\crv{"i"+(2,-2)&"j"+(-2,2)};
"g"*{\hole};"k"*{\hole} **\crv{"g"-(0,2)&"k"+(-3,1)};
"k"*{\hole};"b3" **\crv{"k"+(3,-1)&"b3"+(0,3)} \POS?(.3)="m" \POS?(.6)="l";
"j"*{\hole};"l"*{\hole} **\crv{"j"+(2,-2)&"l"+(-3,0)} \POS?(.4)="n";
"h";"m"*{\hole} **\crv{"h"-(0,2)&"m"+(2,2)};
"m"*{\hole};"n"*{\hole} **\crv{"m"-(2,2)&"n"+(1,2)};
"n"*{\hole};"b2" **\crv{"n"-(1,2)&"b2"+(0,2)};
"l"+(12,-1)="o";
"l"*{\hole};"o"*{\hole} **\crv{"l"+(2,0)&"o"+(-2,0)};
"c"*{\hole};"o" **\crv{"c"-(-3,2)&"o"+(-1,3)};
"o";"t6" **\crv{"o"+(4,-12)&"t6"-(0,5)} \POS?(.6)*\dir{<};
"a"+(6,-4)="p";
"a";"p"*{\hole} **\crv{"a"+(3,0)&"p"+(-2,2)};
"p"*{\hole};"o"*{\hole} **\crv{"p"+(2,-2)&"o"+(4,0)} \POS?(.5)="q";
"a"*{\hole};"q"*{\hole} **\crv{"a"+(1,-2)&"q"+(-4,0)} \POS?(.45)="r";
"q"*{\hole};"p" **\crv{"q"+(6,0)&"p"+(4,2)};
"p";"r"*{\hole} **\crv{"p"-(2,1)&"r"+(2,1)};
"r"*{\hole};"c" **\crv{"r"-(2,1)&"c"+(2,0)};
"t5";"t7" **i\crv{"t5"+(0,-10)&"t7"+(0,-10)} \POS?(.42)="s";
"t5";"s"*{\hole} **\crv{"t5"+(0,-5)&"s"-(5,0)};
"s"*{\hole};"t7" **\crv{"s"+(3,0)&"t7"+(0,-10)} \POS?(.5)*\dir{>}
\endxy
\]

Freyd--Yetter~\cite{FreydYetter1989}  showed that $\TT$ is an autonomous braided monoidal category with a ``freeness property''. 
We can see that $\BB$ is the braided monoidal subcategory of $\TT$ by identifying $n \in \BB$ with  $\underbrace{+ \ldots +}_n \in \TT$.   
The braids occur as the progressive endomorphisms in $\TT$.   
The endomorphisms of the unit object 0 are precisely the \emph{oriented links}.

\begin{Proposition}
The trace of a braid in $\TT$ is the link obtained as the Markov closure.
\end{Proposition}
Here are examples of the trace of a braid, the second one is a trefoil.

\vskip-1.5cm\[
\xygraph{!{/r1.1pc/: /u1.2pc/::}
[uu] ( ?*+!D\txt\tiny{+}  !{+U\vcap[7]} [r(7)]="a7" , !{\vtwistneg}
( !{\xcapv-@(0)} []*+!U\txt\tiny{+} !{+D\vcap[-7]} [r(7)] - "a7"
 , [r] !{\vtwistneg} ( 
   ?*+!U\txt\tiny{+} !{+D\vcap[-5]} [r(5)] !{\POS= "a5"}
   , [r]*+!U\txt\tiny{+} !{+D\vcap[-3]}  [rrr]="a3"
 ) , [urrr] (
     ?*+!D\txt\tiny{+} !{+U\vcap} [r] !{\POS= "a1"}
    , [l]*+!D\txt\tiny{+} !{+U\vcap[3]} [rrr]- "a3"
    , [ll]*+!D\txt\tiny{+} !{+U\vcap[5]} [r(5)]- "a5"
 , [l]!{\vtwistneg}
   [r]!{\xcapv@(0)} []*+!U\txt\tiny{+} !{+D\vcap[-]} [r]- "a1"
)))}
\qquad\qquad\qquad\qquad
\xygraph{ !{0;/r1pc/:/u1.25pc/::} [uu]
 ( !{\vtwist\vtwist\vtwist}[uuu] 
   (?*+!D\txt\tiny{+} !{+U\vcap[3]}[rrr]-[d(4.45)] 
   , [r]*+!D\txt\tiny{+} !{+U\vcap}[r]-[d(4.45)]  )
 , [ddd]*+!U\txt\tiny{+} !{+D\vcap[-3]} 
 , [dddr]*+!U\txt\tiny{+} !{+D\vcap[-1]}
 )}
\]
\vskip-35pt
\noindent
Each tensor-preserving (= strong monoidal) functor
\[
F : \TT   \longto \Vect
\]
will produce a complex number for each oriented link $\ell$ :
\[
0 \stackrel{\ell}{\longto}  0   \longmapsto \CC  \stackrel{F\ell}{\longto} \CC\,.
\]
Such $F$ are determined by a single vector space $V$ and a special kind of Yang--Baxter operator
\[
y : V \otimes V \cong V \otimes V,
\]
on $V$ :
\[
F (+ + - + -) = V \otimes V \otimes V^\ast \otimes V \otimes V^\ast.
\]
There are techniques for solving the Yang--Baxter equation to find such functors, and hence, invariants of links; 
see Turaev~\cite{Turaev1988}.

\subsection{Trace without duals}

A (\emph{right}) \emph{internal hom} $C^A$ for objects $A$ and $C$ of a monoidal category $\catC$ is an object equipped with a morphism (called ``evaluation'')
\[
\ev \ : A \otimes C^A \longto C
\]
which induces a bijection
\begin{align*}
\catC (B,C^A) &\cong \catC (A \otimes B, C)
\\
(B  \stackrel{f}{\longto} C^A) 
  &\longmapsto (A \otimes B \xylongto{1 \otimes f} 
  A \otimes C^A \stackrel{\ev}{\longto}   C)\,. 
\end{align*}
However, internal homs can exist without duals (e.g., in $\catC = \Vect $ and $ \Set $).

\goodbreak
\medskip\noindent
I know of no extension of the string notation, in keeping with the geometry as for duals, which covers internal homs.

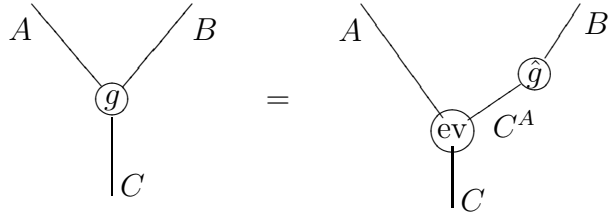
\goodbreak
\[
\xygraph{ !{/r2.5pc/: /u3pc/::}
 [d] *+[Fo][o]{g} (
  ? - [ul]^(.9)*{A} , ? - [ur]_(.9)*{B} , ? - [d]^(.9)*{C}
 )} \quad = \quad  
 \xygraph{ !{/r2.85pc/: /u2pc/::}
 [d(1.5)] *+[Fo][o]{\ev} (
  ? - [uul]^(.9)*{A} 
  , ? - [ur(.9)]*+[Fo][o]{\hat{g}} _*{C^A}
     - [r(.5)u(1.1)]_(.9)*{B}
  , ? - [d(1.2)]^(.9)*{C}
 )} 
\]
From field theory in physics there are geometric situations where internal homs can exist without duals.   
This led Stolz--Teichner~\cite{StolzTeichner} to develop a notion of \emph{trace} using only the ``weak duals'' $I^A$.

The following treatment of ``contraction'' is adapted from these authors. 
A quadruple $(X,Y,A,B)$ of objects of a monoidal category $\catC$ is called \emph{approximable} when $Y^A$ exists and the following composite is injective.
\[
\xymatrix @C+8mm {
\catC (X,Y^A \otimes B) \ar[r]^-{A \otimes -} 
& \catC (A \otimes X, A \otimes Y^A \otimes B)  \ar[r]^-{\catC(1,\ev \otimes 1)} 
& \catC (A \otimes X, Y \otimes B)
}\]
A morphism $f : A \otimes X \longto Y \otimes B$ is called \emph{nuclear} when it is in the image of that composite.

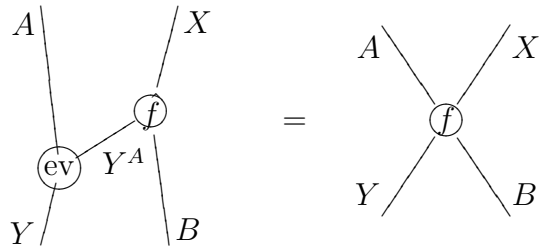
\[
\xygraph{ !{/r2.85pc/: /u3pc/::}
 [d(1.5)] *+[Fo][o]{\ev} (
 ? - [l(.2)u(1.7)]^(.9)*{A} , ? - [l(.2)d(.8)]_(.9)*{Y} , 
 -  [u(.6)r] *+[Fo][o]{\hat{f}}_*{Y^A} (
  ? - [r(.2)d(1.4)]^(.9)*{B}, ? - [r(.3)u(1.1)]_(.9)*{X}
 ))} \quad = \quad  
 \xygraph{ !{/r2pc/: /u3pc/::}
 [d(1)] *+[Fo][o]{f} (
 ? - [ul]^(.9)*{A} , ? - [ur]_(.9)*{X} , ? - [dl]_(.9)*{Y} , ? - [dr]^(.9)*{B}
  )} 
\]

\begin{Example}
If $A$ has a right dual then any $(X,Y,A,B)$ is approximable. 
Indeed, the composite is invertible so every $f$ is nuclear.
\end{Example}

\noindent
The \emph{trace} (or \emph{contraction by} $A$) of 
\[
f : A \otimes X \longto Y \otimes A
\]
in a braided monoidal  $\catC$ is the composite
\[
\Tr (f) : X \stackrel{\hat f}{\longto} Y^A \otimes A \stackrel{\gamma}{\longto} A \otimes Y^A \stackrel{\ev}{\longto} Y\,.
\]

Let $\Tv$ be the category of topological complex vector spaces which are locally convex, complete and Hausdorff. The tensor product represents continuous bilinear maps.

\begin{Theorem} [Stolz--Teichner \cite{StolzTeichner}] 
Take $\catC = \Tv$.
\begin{enumerate}[\upshape 1.]
\item
 If $A$ has the approximation property (from analysis) then $( \CC,  \CC, A, A)$ is approximable.
\item
 A morphism $f:A\longto B$ is nuclear as in analysis if and only if $f:A \otimes \CC \to \CC \otimes B$ is nuclear.
\item
 If $A$ has the approximation property and $f:A\to A$ is nuclear then the trace $\Tr(f)$ is the classical trace.
\end{enumerate}
\end{Theorem}

\section{Monoidal categories, Hall algebras and representation theory\label{Lecture2}}

This lecture is dedicated to my colleague and friend Brian Day.

\noindent
We shall discuss the following three processes.
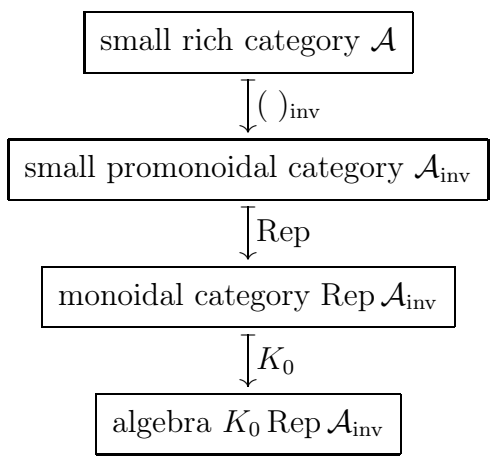
\[
\xygraph{ !{\UseTips{}}!{/u4pc/::}
 []*++[F-]+!C\txt{small rich category $\catA$}
  :@{|->}^{\objectstyle(\;)_{\mathrm{inv}}}
 [d]*++[F-]+!C\txt{small promonoidal category $\catA_{\inv}$}
  :@{|->}^{\objectstyle\Rep}
 [d]*++[F-]+!C\txt{monoidal category $\Rep\catA_{\inv}$}
  :@{|->}^{\objectstyle K_0}
 [d]*++[F-]+!C\txt{algebra $K_0\Rep\catA_{\inv}$}
}\]

\subsection{Preliminary remarks}

Certainly the representations of a given abstract group form a category.  
However, many concrete groups naturally occur in families: symmetric groups $\LieS_n$, $n \in \NN$; 
braid groups $\BB_n$,  for $n \in \NN$; general linear groups $\GL_{\FF}(n)$ over a fixed field $\FF$; and so on.  
Such families arise from a category $\catA$ as the automorphism groups $\Aut_{\catA}(A)$, $A \in \ob \catA$.  
More categorically, we consider the groupoid $\catA_{\inv}$ with the same objects as $\catA$ but only the invertible morphisms.

A goal of this lecture is to show how studying the category of representations of $\catA_{\inv}$ can enhance the study of representations of the individual groups $\Aut_{\catA}(A) = \catA_{\inv}(A,A)$.

For example, take $\catA$ to be the category of finite sets and functions between them.  Let $\catSS=\catA_{\inv}$  denote the category of finite sets and bijective functions between them.   This groupoid is monoidal with disjoint union (sum $+$) as tensor product.  Mac Lane~\cite{ML1963} used (a skeleton of) this category to express the coherence theorem for symmetric monoidal categories.

As a groupoid $\catSS$ is equivalent to the disjoint union of the symmetric groups
\[
\LieS_n \cong \catSS (\langle n \rangle , \langle n \rangle) \text{ where } \langle n \rangle = \{1,2,\ldots,n\}\,.
\]
However, the monoidal structure includes the important functions
\[
+ : \LieS_m \times \LieS_n  \longto \LieS_{m+n}\,.
\]

As with any category, the diagonal functor $\delta : \catSS \to \catSS \times \catSS$ is a comonoidal structure.   
In this case, $\delta$ preserves the monoidal structure.

In passing from $\catA$ to $\catA_{\inv}$, many of the good properties $\catA$ may have are lost.    
For example, while $+$ is the coproduct in finite sets and functions, it is not the coproduct in $\catSS$.    
In fact, $\catSS$ has little in the way of limits or colimits, or internal homs.

The work of Day~\cite{DayConv} shows how to complete a monoidal category 
(and, more generally, a promonoidal category) to have both colimits and internal homs.

\subsection{Promonoidal categories\label{Promon}}

A \emph{promonoidal structure} on a category $\catC$ consists of functors
\begin{align*}
P : \catC^{\op} \times \catC^{\op} \times \catC &\longto \Set 
\\
J : \catC &\longto \Set 
\end{align*}
and natural isomorphisms
\begin{align*}
&\alpha : \int^X P (X,C\,;\,D) \times P(A,B\,;\,X) \cong \int^Y P(A,Y\,;\,D) \times P (B,C\,;\,Y)
\\
&\lambda : \int^X P(X,A\,;\,B) \times JX \cong \catC (A,B)
\\
&\rho : \int^X P (A,X\,;\,B) \times JX \cong \catC(A,B)
\end{align*}
satisfying two coherence conditions.

\begin{example}\label{ProEx1}%
For any category $\catC$, put 
\begin{align*}
P (A,B\,;\,C) &= \catC (A,C) \times \catC (B,C)
 \\
JA &= 1
\end{align*}
(which corresponds to the comonoidal structure $\delta : \catC \to \catC \times \catC$ on $\catC$).
\end{example}%

\begin{example}\label{ProEx2}%
For any monoidal category $\catC$, put
\begin{align*}
P (A,B\,;\,C) &= \catC (A \otimes B,C)
\\
JA &= \catC (I,A)\,.
\end{align*}
\end{example}%

\begin{example}\label{ProEx3}
\fbox{$\catC : =  \Aut \catG$} \quad
For any groupoid $\catG$, take $\catC$ to have objects $(A,s)$ 
where $s : A \to A$ in $\catG$ and morphisms $f : (A,s) \to (B,t)$ the morphisms $ f : A \to B$ in $\catG$ satisfying $f s = t f$.    
Put
\[
P\bigl((A,s),(B,t)\,;\,(C,r)\bigr) = 
 \bigl\{ A \stackrel{u}{\longto} C \stackrel{v}{\longfrom} B \text{ in } \catG \bigm|  \conj{u}{s} \,\conj{v}{t} = r \bigr\}
\]
where $\conj{u}{s} = u s u^{-1}$, and
\[
J (A,s) = \begin{cases}
               1 &\text{when } s = 1_A \\
\emptyset &\text{otherwise}. \end{cases}
\]
\end{example}%

\begin{example}\label{ProEx4}%
Let $\catA$ be an abelian category.   
Take $\catC$ to have the same objects as $\catA$ but only the invertible morphisms. 
Put
\[
P(A,B\,;\,C) = \bigl\{(f,g) \bigm |  0 \to A \stackrel{f}{\to} C \stackrel{g}{\to} B \to 0
\text{ is a short exact sequence in } \catA  \bigr\}
\]
and
\[
JA = \begin{cases}
               1 & \text{for } A = 0 \\
\emptyset &\text{otherwise}.
\end{cases}
\]
The natural isomorphism $\alpha$ is defined by
\[
\xymatrix @C+5mm {
[(u,v),(f,g)] \ar @{<->} [r] |-*@{|} & [(f',g'),(u',v')] }
\]
as in the following $3 \times 3$ diagram of short exact sequences in $\catA$. 
\[
\xymatrix{
& 0 \ar[d] & 0 \ar[d] & 0 \ar[d] & \\
0 \ar[r] & A \ar[d]_-{f} \ar[r]^-{1} & A \ar[d]_-{f'} \ar[r] & 0 \ar[d] \ar[r] & 0 \\
0 \ar[r] & X \ar[d]_-{g} \ar[r]^-{u} & D \ar[d]_-{g'} \ar[r]^-{v} & C \ar[d]^-{1} \ar[r] & 0 \\
0 \ar[r] & B \ar[d] \ar[r]_-{u'} & Y \ar[d] \ar[r]_-{v'} & C \ar[d] \ar[r] & 0 \\
& 0 & 0 & 0 &
}
\]
\end{example}%
\goodbreak
\goodbreak

\subsection{Day convolution}

Let $\catC$  be a promonoidal (small) category and let $\catV$ be a symmetric monoidal closed complete and cocomplete category.   Then let $[\,\catC, \catV\,]$ denote the category of functors $M : \catC \longto \catV$ and natural transformations.

For a set $\Lambda$ and object $V$ of $\catV$, we write $\Lambda \cdot V$ for the coproduct of $\Lambda $ copies of~$V$.   
We write $[V,W]$ for the left ($=$ right) internal hom of $\catV$.

The \emph{convolution tensor product} on $[\,\catC,\catV\,]$ is defined by
\[
\mathop{(M \ast N)} C = \int^{A,B} \ P(A,B\,;\,C) \cdot MA \otimes NB
\]
with unit
\[
JC = JC \cdot I\,.
\]
Both left $\laction{N}{L}$ and right $L^M$ internal homs exist:
\begin{align*}
\mathop{(\laction{N}{L})} A &= \int_{B,C} [P(A,B\,;\,C) \cdot NB,LC\,]
\\
\mathop{(L^M)} B &= \int_{A,C} [P(A,B\,;\,C) \cdot MA,LC\,]
\end{align*}
When $\catV = \Set $, convolution is just the left Kan extension.
\[
\xymatrix{\ar @{} [drr] |(.45)*{\stackrel{\cong}{\Longrightarrow}}  
\catC^{\op} \times \catC^{\op} \ar[rd]_-{\overline{P}}\ar[rr]^-{y\times y}  
  && [\,\catC,  \Set ] \times [\,\catC, \Set ] \ar[ld]^-{-\ast-} \\
& [\,\catC, \Set ]  &
}\]

\subsection{Hadamard product}

For any category $\catC$, the comonoidal example gives the convolution structure on $[\,\catC,\catV\,]$ as
\begin{align*}
\mathop{(M \underline{\times} N)} C 
 &= \int^{A,B} \bigl(\catC (A,C) \times \catC(B,C)\bigr) \cdot MA \otimes NB
\\
&\cong MC \otimes NC\,,
\end{align*}
the Hadamard (or pointwise) tensor  product.

\subsection{Cauchy product}

For any monoidal category $\catC$, Example~\ref{ProEx2} gives the convolution structure on $[\,\catC,\catV\,]$ as
\[
\mathop{(M \bullet N)}C = \int^{A,B} \catC (A\otimes B,C) \cdot MA \otimes NB
\]
which might be called the Cauchy product.

For the case $\catC=\catSS$, we have the simplification
\[
\mathop{(M \bullet N)} C \cong \sum_{S \subseteq C}  MS \otimes N (C \backslash S)\,,
\]
where $C \backslash S$ is the complement of the subset $S$ of $C$ in $\catC$.

\subsection{Species and the substitution operation}

A functor $M : \catSS \to  \Set $ is a \emph{species} in the sense of Joyal.   
A functor  $M : \catSS \to  \Vect $ is a \emph{vector species}.    
On species, apart from the Hadamard and Cauchy products, there is a substitution product.
\[
\mathop{(M \circ N)} C = \int^A MA \otimes N^{\bullet \# A}   C
\]
where $N^{\bullet n} = \underbrace{N \bullet \ldots \bullet N}_n$  is the $n$-fold Cauchy product.    
This is a monoidal structure on $[\catSS,\catV\,]$ but does not come by convolution 
(since $M  \circ -$ does not preserve colimits).

\subsection{Symmetric group representations}

$\Rep\catSS$ denotes the full subcategory of $[\catSS, \Vect ]$ consisting of those $M$ satisfying $\dim MA < \infty$ 
and $= 0$ for $\# A$ large enough.

$\Rep\catSS$ is equivalent to the weak product of the usual categories $\Rep\LieS_n$ of representations of the permutation groups $\LieS_n$ with $n \geq 0$.    
The irreducible representations of the individual $\LieS_n$ give all the irreducible objects in $\Rep\catSS$ up to isomorphism.

A {\em class function} on $\catSS$ is a functor $f : \Aut \catSS \to \CC$ (where $\CC$ is discrete).    
The Grothendieck ring 
\[
\ringK_0 \Rep\catSS
\] 
is isomorphic to the ring of class functions where multiplication is
\[
(f \cdot g) (A, \sigma) = \sum_{S \subseteq A \atop \sigma S = S} 
  f(S, \sigma \upharpoonright_S) \times g(A\backslash S, \sigma \upharpoonright_{A\backslash S})\,.
\]
This ring is isomorphic to \emph{the ring of symmetric functions}.

\subsection{Finite general linear group representations} 

Let $\FF_q$ be a field with $q$ elements
and let $\catA_q$ be the abelian category of finite vector spaces over $\FF_q$ and linear functions between them.   
Now let $\catGL_q$ denote the groupoid with the same objects as $\catA_q$ but with only the linear bijections.

We are in the situation of Example~\ref{ProEx4}.   
There is a promonoidal structure defined on $\catGL_q$ by 
\[
P(A,B\,;\,C) = \{A \stackrel{f}{\rightarrowtail}  C  \stackrel{g}{\twoheadrightarrow} B \text{ s.e.s. in } \catA_q\}\,.
\]
Hence there is a convolution monoidal structure on the category
\[
[\catGL_q, \Vect]
\]
defined by
\begin{align*}
\mathop{(M \bullet N)} C &= \int^{A,B} P(A,B\,;\,C) \cdot MA \otimes NB
\\
&\cong \bigoplus_{V \leq C} MV \otimes N(C/V)
\end{align*}
where $V$ runs over the linear subspaces of $C$ in $\catA_q$.  
To understand the isomorphism, take $(f,g) \in P (A,B\,;\,C)$ and put $V = fA \leq C$.    
Then we have
\[
\xymatrix{
A \  \ar[d]_\sigma^\cong  \ar@{>->}[r] &C \ar[d]_1 \ar@{>>}[r] &    B  \ar[d]_\cong^\tau \\
 V \ar@{^{(}->}[r] &  C \ar[r] &     C/V 
 }\]
If $x \in  MA$ and $y \in NB$, the isomorphism takes $[(f,g),x\otimes y]$ to $(M\sigma) x \otimes (N\tau) y$.

\medskip

There is a classical way of assigning a representation $M \bullet N$ of $\GL_q (m+n)$ to a representation $M$ of $\GL_q(m)$ and a representation $N$ of $\GL_q(n)$.   Let $P_{m,n}$ denote the ``parabolic'' subgroup of $\GL_q(m+n)$ consisting of those bijective linear $\rho : \FF_q^{m+n} \to \FF_q^{m+n}$ which take $\FF^m_q$ into itself.    By taking diagonal blocks, we obtain a surjective group morphism $P_{m,n} \to \GL_q (m) \times \GL_q (n)$.

Then $M \bullet N$ is defined as the left Kan extension
\[
\xymatrix{
P_{m,n}  \  \ar@{>>}[d]  \ar@{^{(}->}[r]  \ar@{}[rd]|*{\Longrightarrow}
&   \GL_q (m+n)  \ar@{-->}[d]^{M \bullet N}  & \\
 \GL_q (m) \times \GL_q (n)  \ \ar[r]_-{M \underline{\times} N} &   \Vect\,;
 }
 \]
that is, $M \bullet N$ is induced by the restriction of $M \  \underline{\times} \  N$ to $P_{m,n}$.   
This agrees with the convolution structure on $[\catGL_q, \Vect]$ for these special objects $M$ and $N$. 

Green~\cite{Green1955} defined a ring structure on the class functions $f: \Aut\catGL_q \to \CC$ where
\[
(f \bullet g) (C,\sigma) = \sum_{V \leq C \atop \sigma V = V} f(V,\sigma \upharpoonright_V) \times g (C/V,\sigma/V)\,.
\]
It is easy to see that this ring is commutative.

Write $\Rep\catGL_q$ for the full subcategory of $[\catGL_q,\Vect]$ consisting of the $M$ whose values are finite dimensional and have $MA = 0$ for $A$ of  large enough dimension. This is a categorical version of Green's ring of class functions.

\medskip

The question arises as to how commutativity of the ring of class functions is reflected at the categorical level of $\Rep\catGL_q$.
\begin{Theorem}[Joyal--Street \cite{GLFq}]
$\Rep\catGL_q$ with monoidal structure $M\bullet N$ is braided.
\end{Theorem}

\noindent
I will explain this. A \emph{lax braiding} for a promonoidal category $\catC$ is a natural family
\[
\gamma^{}_{A,B\,;\,C} : P(A,B\,;\,C) \longto P(B,A\,;\,C)
\]
such that the natural family
\[
\gamma^{}_{M,N} : M \ast N \longto N \ast M
\]
defined as the composite
\[ \xymatrix @C+5mm{
\int^{A,B} P(A,B\,;\,-)  \cdot   MA \times NB    \ar[r]^{\int \gamma \cdot  1 \times 1}
 & \int^{A,B} P(B,A\,;\,-)  \cdot MA \times NB   \ar[d]_\cong^{\int 1 \cdot \gamma}
\\
& \int^{A,B} P(B,A\,;\,-)  \cdot NB \times MA
}\]
satisfies the conditions for a braiding on $[\,\catC, \Set ]$ excluding invertibility.    
If each $\gamma^{}_{A,B\,;\,C}$ is invertible it is called a \emph{braiding} on the promonoidal category $\catC$.    
In fact, there is a bijection between (lax) braidings on $\catC$ and (lax) braidings on $[\,\catC,\Set ]$.

Yet, it may happen that $[\,\catC,\catV\,]$ is (lax) braided without $\catC$ (lax) braided.

\noindent
The functor
\[
-  \cdot  I : \Set  \longto \catV
\]
takes cartesian product to tensor product in $\catV$. 
Using this, we see that (lax) braidings for $[\,\catC,\catV\,]$ come from natural families
\[
\gamma^{}_{A,B\,;\,C} : P(A,B\,;\,C)   \cdot I \longto
P (B,A\,;\,C) \cdot  I.
\]
When $\catC$ arises from an abelian category $\catA$ as in Example~\ref{ProEx4}, 
there is a lax braiding on $[\,\catC, \Vect]$ coming from the linear function
\[
\gamma : P(A,B\,;\,C) \cdot  \CC   \longto  P(B,A\,;\,C) \cdot  \CC 
\]
represented by the matrix
\[
\gamma : P (A,B\,;\,C) \times P(B,A\,;\,C) \longto \CC 
\]
defined by
\[
\gamma^{}_{(f,g),(h,k)}   = \begin{cases}
1 & \text{when } \ f,g,h,k \text{ form a direct sum (as below)}
\\
0 & \text{otherwise}
\end{cases}
\]

\paragraph*{\it Direct sum conditions\/\upshape:}
$ \xymatrix @C+5mm{
 A \ar @<3pt> [r]^f  
 & C \ar @<3pt> [r]^g \ar @<3pt>[l]^k
 & B \ar @<3pt> [l]^h
}$
such that $ k \circ f = 1_A$, $g \circ h = 1_B$ and  $f \circ k + h \circ g = 1_C$.

\medskip\noindent
The matrices $\gamma$ can be huge with many zero entries.    
For general $\catA$ we do not expect the matrices to be invertible.    
As a $\Vect$-promonoidal category, $\catC$ is only lax braided in general.

\medskip
Assume each object of $\catA$ has only finitely many subobjects (as in the case of~$\catA_q$).   
Then convolution on $[\,\catC,\Vect]$ is defined by
\[
\mathop{(M \bullet N)}C = \bigoplus_{V \leq C} MV \otimes N(C/V)\,.
\]
The lax braiding for convolution induced by that on $\catC$ is
\[
\xymatrix{\mathop{(M \bullet N)} C   \ar[d]^{\cong}   \ar[r]^-{\gamma^{}_{M,N} C} & (N \bullet M)C  \ar[d]^{\cong} \\
 \bigoplus\limits_{U\leq C} MU \otimes N (C/U)  \ar[r]^-\theta & \bigoplus\limits_{V\leq C} M(C/V)\otimes NV}
\]
where the matrix $\theta$ has entries
\[
\theta_{U,V} = \begin{cases}
M(r_{U,V})\otimes N(s_{U,V}) & \text{when } C= U \oplus V \\
      0 & \mbox{otherwise}.
\end{cases}
\]
Here an (internal) direct sum $C= U \oplus V$ induces canonical isomorphisms
\[
r_{U,V} : U \cong C/V \quad \mbox{and} \quad s_{U,V} : C/U \cong V\,.
\]
Ringel~\cite{Ringel1990} defines, for an $\FF_q$-linear abelian category $\catA$ of finite global dimension,
\begin{itemize}
\item 
the \emph{multiplicative Euler form}
\[
\langle A,B \rangle^2_m = \frac{\catA(A,B) \ \Ext^2 (A,B) \ \Ext^4 (A,B)\ \ldots}{\Ext^1(A,B) \ \Ext^3(A,B) \ \Ext^5(A,B) \ \ldots}\,;
\]
\item 
the vector space $\Hall_{\catA}$ of functions
$ f : (\ob \catA)/_{\cong} \longto \CC $ of finite support;
\item 
the product
\[
(f  \bullet g) [C\,]
 = \sum_{V \leq C} \langle V,C/V \rangle_m f(V) \cdot g (C/V)\,.
\]
\end{itemize}

Ringel proves this is an associative algebra. 
Green~\cite{Green1995} defines a comultiplication on $\Hall_\catA$ but has to modify (topologically) $\otimes$ on $\Vect$. 
Xiao~\cite{Xiao1997} produces an antipode for $\Hall _\catA$. For further details see Schiffmann's lectures~\cite{Schiff}.

There are modified versions of $\Hall_\catA$ for $\catA$ a triangulated category; 
the distinguished triangles are used in place of exact sequences to define $P(A,B\,;\,C)$.

It would be interesting to study the connection between the Drinfeld double of Xiao's Hall Hopf algebra and the monoidal centre of $[\,\catC,\Vect ]$.  
(There will be more about monoidal centres in Sections~\ref{Lecture3} and~\ref{Lecture4}).  

\bigskip
We return to the case of $\catA_q$ and the braiding  on $\Rep \catGL_q$.  
To prove the invertibility of the braiding we must delve more deeply into the representation theory of the groups $\GL_q (n)$.   
We aim at emphasising categorical  features.     

\medskip\noindent
\emph{Irreducible representations} are those which are indecomposable under direct sum.

\medskip\noindent
\emph{Cuspidal representations} are those which are indecomposable under the tensor product $M \bullet N$. 

\medskip\noindent
More precisely, define $M\in\Rep\catGL_q$ to be \emph{trivial} when $MA=0$ for all $A$ with $\dim A > 0$.   
Define $M\in \Rep \catGL_q$ to be  $\bullet$-\emph{indecomposable} when 
\[
M \cong N \bullet L \;  \Rightarrow \;  N  \text{ or }  L \text{ trivial}.
\]
Define $M$ to be \emph{cuspidal} when it is both irreducible and $\bullet$-indecomposable.  
This agrees with the definition due to  Harish-Chandra (in terms of unipotent radicals) since an irreducible in $\Rep \catGL_q$  is isomorphic to a representation of a single $\GL_q(n)$.   
By naturality of $\gamma^{}_{M,N}$, it suffices to prove invertibility for $M$ and $N$ irreducible.    
Using the braid conditions for $\gamma^{}_{M \bullet N,L}$ and $\gamma^{}_{M,N\bullet L}$, we see that it suffices to prove invertibility for $M$ and $N$ cuspidal.

\medskip
After detailed study of iterated tensor products
\[
(M_1\bullet \ldots \bullet M_n) C \;\; = \bigoplus_{0\leq V_1 \leq \ldots V_{n-1} \leq C} M_1V_1 \otimes M_2 (V_2/V_1) \otimes \ldots \otimes M_n (C/V_{n-1})
\]
we obtain:
\begin{Theorem} [Joyal--Street \cite{GLFq}] 
Suppose $M$ and $N$ are cuspidal in $\Rep \catGL_q$ with $MA\neq 0$ and $NB \neq 0$ for $\dim A = r$ and $\dim B = s$.    

\goodbreak
\begin{enumerate}[\upshape(i)]
\item
 If $M \not\cong N$ then \ 
$ \gamma^{}_{N,M} \circ \gamma^{}_{M,N} = q^{rs} \ 1_{M\bullet N}\,$.
\item
  If $M = N$ then \ 
$\gamma^{}_{M,M} \circ \gamma^{}_{M,M} = q^{r(r-1)/2}  (q^r-1) \ \gamma^{}_{M,M} + q^{r^2}\ 1_{M\bullet M}\,$.
\end{enumerate}
\end{Theorem}
\noindent
The invertibility of the braiding on $\Rep \catGL_q$ then follows.

\subsection{Hecke algebroids}

Let $\catX$ be a monoidal category.  A \emph{Yang--Baxter operator} $y$ on a family of objects $X_s$, $s\in S$, of $\catX$ is a family of isomorphisms
\[
y_{s,t} : X_s \otimes X_t \to X_t \otimes X_s
\]
such that 
\[ \xymatrix{  
 & X_t \otimes X_s \otimes X_u  \ar[r]^-{1\otimes y} &  X_t \otimes X_u \otimes X_s \ar[rd]^-{y\otimes 1} \\
X_s \otimes X_t \otimes X_u \ar[ru]^-{y \otimes 1} \ar[rd]_-{1\otimes y} &&&  X_u \otimes X_t \otimes X_s\;. \\
 & X_s \otimes X_u \otimes X_t    \ar[r]_-{y\otimes 1} & X_u \otimes X_s \otimes X_t    \ar[ru]_-{1\otimes y} 
 }\]
When $\catX$ is braided, $y_{s,t} = \gamma^{}_{X_s,X_t}$ is an example.

Suppose we are given functions
$d : S \times S \to \CC \backslash \{0\}$ and  $e : S \to \CC$,
with $d$ symmetric. 
The \emph{Hecke algebroid} $\catH (d,e)$ is the $\CC$-linear (i.e. enriched in $\Vect $) strict monoidal category universally generated by a family of objects indexed by $S$ bearing a Yang--Baxter operator satisfying
\begin{align*}
y_{ts} y_{st} &= d(s,t) \,1 \;\text{ for } s \neq t
\\
y_{ss} y_{ss} &= e(s)\, y_{ss} + d(s,s) \,1\,.
\end{align*}

\begin{Theorem}[Joyal--Street \cite{GLFq}] 
$\Rep \catGL_q$ is the projective {\upshape(}$= \Vect $-enriched Cauchy{\upshape)} completion of a Hecke algebroid.
\end{Theorem}

The objects of the Hecke algebroid are isomorphism classes of cuspidal representations 
(which are in bijection with irreducible monic polynomials in $\FF_q[x]$ with non-zero constant term).  
We use results of Gelfand and Graev~\cite{GG1962}, Gelfand~\cite{Gel1970}, Harish-Chandra~\cite{H-C1970}
and Macdonald~\cite{Macdon1995}.

 Let us investigate whether we can lift the comultiplication
 \[
 \Hall_\catA \to \Hall_\catA \mathbin{\tilde{\otimes}} \Hall_\catA 
 \]
 of Green to the categorical level. The promonoidal multiplication
\[
P : \catC^{\op} \times \catC^{\op} \times \catC \to \Set\,,
\]
on the groupoid $\catC$ coming from $\catA$, is defined on morphisms
\[
A' \stackrel{u}{\longto} A \,, \quad 
B' \stackrel{v}{\longto} B \,, \quad 
C \stackrel{w}{\longto}  C'
\]
by
\[
\mathop{P(u,v\,;\,w)}(f,g) = (wfu,v^{-1}gw^{-1})
\]
\[
\xymatrix{
A \    \ar[r]^f &C \ar[d]_w \ar[r]^g &    B  \\
\ar[u]^u A' \ \ar[r]_{f'} &  C' \ar[r]_{g'} & B'  \ar[u]_v
 }\]
Since $\catC$ is a groupoid, $\catC^{\op} \cong \catC$.    
Therefore, we have 
\[
\overline P  : \catC^{\op} \times \catC \times \catC \to \Set 
\]
defined by
\[
\overline P (C\,;\,A,B) = P(A,B\,;\,C)
\]
and
\[
 \mathop{\overline P  (w\,;\,u,v)}(f',g') = (w^{-1},f'u^{-1},vg'w)\,.
\]

A functor of the form $T : \catA^{\op}  \times \catB \to \Set $ is called a module (profunctor or distributor)
\[
\catA \stackrel{T}{\longto} \catB
\]
from $\catA$ to $\catB$, so that, for modules $\catA \stackrel{T}{\longto} \catB \stackrel{S}{\longto} \catC$, 
we have
\[
\catA \stackrel{S\circ T}{\longto} \catC
\]
defined by
\[
\mathop{(S\circ T)}(A,C) = \int^B S (B,C) \times T(A,B)\,.
\]
Morphisms (2-cells) between modules are natural transformations between the $\Set $-valued functors of two variables. Consider the diagram of modules below in which we will define the 2-cell $\psi$. 
\[
\xymatrix{ & \catC \times \catC \ar[ld]_-{\overline P \times \overline P}  \ar[rr]^P  && \catC \ar[rd]^-{\overline P} 
\\
\catC \times \catC \times \catC \times \catC \ar[rrd]^{\cong}_-{1\times \gamma \times 1} 
 \ar@{}[rrrr]_-*{\psi}|-*{\textrm{\rotatebox{30}{$\Longrightarrow$}}} &&& & \catC \times \catC 
\\
&& \catC \times \catC \times \catC \times \catC \ar[rru]_-{P\times P} 
}
\]
The top path is:
\[
\bigl(\overline P \circ P) \bigl((A,B),(C,D)\bigr) = \int^X \overline P (X\,;\,A,B) \times P (C,D\,;\,X) \,.
\]
The bottom path is:
\begin{multline*}
\lefteqn{(P\times P) \circ (1 \times \gamma \times 1) \circ 
   \mathop{(\overline P \times  \overline P)} \bigl((A,B),(C,D)\bigr)} \\
= \int^{T,U,V,W} P(T,U\,;\,C) \times P(V,W\,;\,D) \times \overline P  (A\,;\,T,V) \times \overline P (B\,;\,U,W) \,.
\end{multline*}
The definition of $\psi$ involves diagrams (i) and (ii), as follows.

\noindent
\begin{inparaenum}[(i)]
\item 
$\xymatrix{
 & C  \ar[d]^u 
\\
A \ \ar[r]^f & X \ar[d]^v \ar[r]^g & B 
\\
 & D & & 
}$ 
\qquad\qquad
\item 
$  \xymatrix{
T \ar[d]_u \ar[r]^{f'} &C \ar[r]^{g'} &U\ar[d]^{u''}
\\
A\ar[d]_{v'} &&B  \ar[d]^{v''}
\\
V \ar[r]_{f''}  &D  \ar[r]_{g''} & W
\\
}$ 
\end{inparaenum}

\noindent
Given $T,U,V,W$, put $X = T \oplus U  \oplus  V  \oplus  W$ so that we have a relation
\begin{eqnarray*}
P(T,U\,;\,C) \times P(V,W\,;\,D)   \times  \overline P (A\,;\,T,V) \times \overline P (B\,;\,U,W)
\\
 \psi \ {-\hspace{-2.5mm}\raisebox{-0.3mm}{$\downarrow$}}  \hspace{40mm}
\\
\overline P (X\,;\,A,B)  \times P(C,D\,;\,X)   \hspace{20mm}
\end{eqnarray*}
defined by
\[
\psi_{\mathrm{(ii),(i)}} =  \begin{cases}
1 & \text{when (i) inside (ii) gives a $3 \times 3$ diagram of s.e.s.}
\\
0 & \text{otherwise}.
\end{cases}
\]
This defines a natural transformation $\psi$ at the $\CC$-linear level.
 
 So $\catC$ becomes ``lax pro-bimonoidal''.   
 It is a categorical version of a bialgebra.   
 This seems worthy of study as a ``categorification'' of the Hall bialgebra.     
 So does the question of an antipode at the categorical level.
 
\subsection{Charades}
 
The sense of the term \emph{charade} intended here is that of Kapranov~\cite{Kapr1995} 
who suggests that the promonoidal structures discussed in (Section~\ref{Promon}) Example~\ref{ProEx4} 
might be used to find a common setting for two different generalizations of class field theory. 
A~related direction of interest is to study the special case $[\catGL_q, \Vect ]$ in connection with combinatorics.  
Joyal species are the $q=0$ case.

\section{Mackey functors and classifying spaces\label{Lecture3}} 

 J.\,A.\,Green has a fundamental paper~\cite{Green1971} in this subject (as he did for Section~\ref{Lecture2}). 
 Other important contributors are A.\,W.\,M.\,Dress~\cite{Dress1973},~\cite{Dress1975}  and S.\,Bouc~\cite{BoucBook}.

In this section we describe the evolution of the definition of Mackey functor itself, 
the structure of the category of Mackey functors~\cite{MFCCC}, 
the centre of the monoidal category of Mackey functors~\cite{Tambara2008}, 
and discuss a connection with the calculation of the classifying spaces of Lie groups. 
(For the last, thanks go to Vincent Franjou for alerting me to~\cite{JackMcC1992}).

\paragraph*{\it Microcosm principle\upshape:\ \ }
A Mackey functor for a group $G$ can simultaneously be regarded as:
\begin{itemize}
\item a generalized representation of $G$, and
\item a decategorified representation theory for $G$.
\[
\Rep G \subseteq \operatorname{Mackey} G \ni  ``K_0 \Rep G\,"
\] 
\end{itemize}

\subsection{The definitions\label{Lecture3sub1}}
  
Left Kan extension $\Lan_i$ along a functor $i : H \to G$ is the functor which is left adjoint to restriction $\Res_i$ along $i$:

\vskip-10pt 
\[ \xymatrix @R-3mm {
[H,\funcM] \ar@<1.5ex>[rr]^{\Lan_i} \ar@{}[rr]|-{\perp} && [G,\funcM] \ar@<1.5ex>[ll]^{\Res_i}  \\
V\circ i && \ar@{|->}[ll] V
}\]
Under smallness conditions on $i$ and with $\funcM$ cocomplete, 
the natural transformation $\lambda : i \circ p \Implies j \circ q$ that should point left to right in the comma category square below 
induces a natural isomorphism
$ \Res_j \circ \Lan_i \cong \Lan_q \circ \Res_p $.

\[ \xymatrix{
 i \downarrow j \ar[d]_p\ar[rr]^q &&K \ar[d]^j  \ar@/^4pc/[dd]^{\Lan_q(W_p)}\\
H\ar[rr]_i  \ar[rrd]_{W}  &&G \ar@{}[r]|*{\cong} \ar[d]^{\Lan_i W} &
\\ 
&&\funcM  & 
}
\]
If a category $D$ is a disjoint union
\[
D = \sum_{\alpha \in \Lambda} D_\alpha
\]
of subcategories $m_\alpha : D_\alpha \to D$ then, for all functors $r : D \to K$, 
we have 
\begin{align*}
 & \qquad \Lan_r \cong \sum_{\alpha \in \Lambda} \Lan_{r \circ m^{}_\alpha} \circ \Res\,.
\\
 & \xymatrix{
  D_\alpha \ar[r]^{m_\alpha} 
 & D \ar[rr]^r  \ar[rd]_{T}  && K \ar[ld]^-{\Lan_r T \; \cong \; \sum_\alpha  \Lan_{r\,m^{}_\alpha}(T\,m^{}_\alpha)} \\
&& \funcM &}
\end{align*}
For a groupoid $D$, each object $d \in D$ gives a group $D(d)=D(d,d)$ under composition.    
Let $\Lambda$ be a representative set of objects in $D$ for all isomorphism classes.   
We have
\[
D \simeq \sum_{d\in\Lambda} D(d)\,.
\]

We apply this to subgroups $i : H \to G$ and $j : K \to G$ of a group $G$, and to $D = i \downarrow j$.    Objects of the groupoid  $D$ are elements $g$ of $G$ while morphisms $(h,k) : g \to g'$ are elements of $H \times K$ such that $kg =g'h$.   Isomorphism classes of objects of $D$ are in bijection with double cosets
\[
KgH = \{kgh \mid k \in K, h \in H\}\,.
\]
Let $[K\backslash G/H] \subseteq G$ represent all double cosets $KgH$ so that we have an equivalence of categories
\[
\sum_{g\in[K\backslash G/H]} (i\downarrow j) (g) \;\simeq \; i\downarrow j\,.
\]
Let $p_g : (i \downarrow j) (g) \to H$ and $q_g : (i \downarrow j) (g) \to K$ be the restrictions of $p$ and $q$.  
By the Kan extension results we have
\[
\Res_j \circ \ \Lan_i \cong \sum \ \Lan_{q_g} \circ \ \Res_{p_g}\,.
\]
Putting $K^g = g^{-1} Kg$ and $^g H = g H g^{-1}$, the restrictions $p_g$ and $q_g$ induce isomorphisms
\begin{align*}
& (i\downarrow j) (g) \cong H \cap K^g \text{ and } (i\downarrow j)(g) =  {}^gH \cap K\,.
\\ 
& \qquad \xymatrix @R-3mm {
 ^g H \cap K\ar[dd]_\leq \ar[dr]^\cong \ar[rr]^\cong_{\gamma_g}   && H \cap K^g \ar[ld]_\cong \ar[dd]^\leq \\
 & (i\downarrow j) (g) \ar[ld]^{q_g} \ar[rd]_{p_g} \ar[dd] & 
\\
K &  & H  
\\
  & i\downarrow j   \ar[lu]^q    \ar[ru]_p} 
\end{align*}
In representation theory, we take $\funcM = \Mod_k$ for a commutative ring $k$ and put
\[
\Res^G_H = \Res_i \text{ and } \Ind^G_H = \Lan_i\,.
\]

\paragraph*{Mackey Decomposition Theorem:}
\[
\Res^G_H \circ \Ind^G_H  \cong 
 \sum_{g\in[K\backslash G/H]} \Ind^K_{{^{g}H\cap K}}  \circ     \Res_{\gamma^{}_g}  \circ \Res^H_{H \cap K^g}\,.
\]

\noindent
This inspires the technical axiom 4 in the following.

\begin{Definition}[Green \cite{Green1971}]
For a group $G$, a \emph{Mackey functor over} $k$ assigns
\begin{itemize}\addtolength{\itemsep}{-7pt}%
\item to each $H \leq G$, a $k$-module $M(H)$,
\item to each $K \leq H \leq G$, a pair of module morphisms
$
t^H_K : M(K) \to M(H)$ and $r^H_K : M(H) \to M(K)$,
\item 
for $H \leq G$ and $g\in G$, a module isomorphism \ 
$
c_{g,H} : M(H) \stackrel{\cong}{\longto} M (^g H)$,
\end{itemize}
satisfying the following four axioms: 
\begin{enumerate}[\upshape 1.]
\item
 $t^H_K \  t^K_L = t^H_L$ \text{ and } $r^K_L \  r^H_K = r^H_L$,
\item
$c_{g',^g H} \ c_{g,H} = c_{g'g,H}$ \text{ and } $c_{h,H} = 1_{M(H)}$,
\item
$c_{g,H} \  t^H_K = t^{g_H}_{g_K} \  c_{g,K}$ \text{ and } $c_{g,K} \  r^H_K = r^{g_H}_{g_K} \  c_{g,H}$
\item
 $r^L_K \  t^L_H =  \sum_{g\in[K\backslash L/H]} \  t^K_{{^{g}H\cap K}} \  c_{g,H\cap K^g} \  r^H_{H\cap K^g}$.
\end{enumerate}
$t^H_K$ is called \emph{transfer}, \emph{trace} or \emph{induction},
$r^H_K$ is called \emph{restriction}, and
$c_{g,H}$ is called a \emph{conjugation map}.
\end{Definition}

\noindent
Suppose $\funcM$ is a cocomplete monoidal category whose tensor preserves colimits in each variable.   Suppose $i : H \to G$ is the inclusion of a subgroup $H \leq G$ and that the categories $[H,\funcM]$ and $[G,\funcM]$ are equipped with the pointwise tensor products.    For $V \in [G, \funcM]$ and $W \in [H,\funcM]$, there is an isomorphism
\[
V \otimes \Lan_i (W) \cong \Lan_i (\Res_i (V) \otimes W).
\]
In the case $\funcM = \Mod _k$, we obtain the following.

\paragraph*{Frobenius Reciprocity:} \quad
$ V \otimes \Ind^G_H (W) \cong \Ind^G_H (\Res^G_H (V) \otimes W)$.

\medskip\noindent
This inspires the technical axiom 6 in the following.
 
\begin{Definition} 
A \emph{Green functor} $A$ for $G$ over $k$ is a Mackey functor equipped with a $k$-algebra structure on each $k$-module $A(H)$, $H \leq G$, satisfying
\begin{enumerate}[\upshape 1.]\setcounter{enumi}{4}
\item 
$t^H_K$, $r^H_K$, $c_{g,K}$ are $k$-algebra morphisms,
\item 
 $a \cdot t^H_K (b) = t^H_K \bigl(r^H_K (a) \cdot b\bigr)$ 
 and $t^H_K (b) \cdot a = t^H_K \bigl(b \cdot r^H_K (a)\bigr)$.
\end{enumerate}
\end{Definition}

\begin{example} 
Each representation $R$ of $G$ determines a Mackey functor $M_R$ with
 \[
 M_R (H) = \GSet(G/H,R) \,.
 \]
\end{example} 
 
\begin{example} 
There is a Green functor $A$ for $G$ over $\ZZ $ with
\[
A(H) = \ringK_0 \Rep_k (H) \,.
\]
\end{example} 

Green himself showed that ``Mackey functors'' could be thought of as a pair of functors $M^\ast$ and $M_\ast$, one contravariant and one covariant, defined essentially on the category $\catC_G$ of connected finite $G$-sets.
 
Dress gave an equivalent definition in terms of $M^\ast$ and $M_\ast$ now extended to the whole category of finite $G$-sets.

Lindner~\cite{Lindner1976} combined $M^\ast$ and $M_\ast$ into a single functor defined on the category of spans between finite $G$-sets. 
For this construction, $\GSet$ can be replaced by any \emph{lextensive} category $\catE$: 
one with finite coproducts, finite limits, and such that the functor
\begin{align*}
\catE/X \times \catE/Y &\longto \catE/X + Y
\\
(U\to X, V \to Y) &\longmapsto (U+V \to X+Y)
\end{align*}
is an equivalence for all $X$ and $Y$.

Write $\Spn (\catE)$ for the autonomous monoidal category for which:
\begin{itemize}
\item 
 objects are those of $\catE$;
\item
 morphisms $[S] : X  \longto 
  Y$ are isomorphism classes of ``spans'' $X \stackrel{u}{ \longfrom} S \stackrel{v}{ \longto} Y$; 
\item
composition is by pullback
$ 
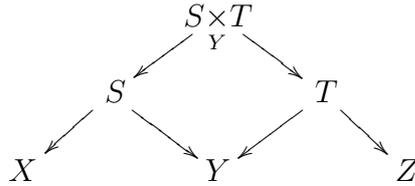
\xymatrix @C-3mm @R-3mm {
 & & *+{S \smash{\operatorname*{\times}\limits_{Y}} T} \ar[ld]  \ar[rd] & & \\
 & \ar[ld] \ S \ar[rd] & & \ar[ld] \,T \ar[rd] & \\
 X & & Y & & Z}$
\item
the tensor product is the cartesian product in $\catE$
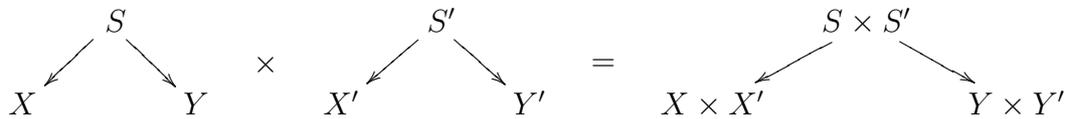
\[
\xymatrix @C-3mm @R-3mm{ 
 & \ar[ld] \ S \ar[rd]&&&& \ S' \ar[ld]  \ar[rd] &&&& S\times S' \ar[ld] \ar[rd]\\
X && Y & \ar@{}[u]|{\mbox{$\times$}} & X' && Y'  & \ar@{}[u]|{\mbox{$=$}}&  X \times X' &  & Y \times Y'}  
\]
\\
(it is not cartesian product in Spn$(\catE)$); 
\item
each object is its own dual
\[
\Spn  (\catE) (X\times Y, Z) \cong \Spn  (\catE) (Y,X \times Z)\,.
\]
\end{itemize}
This construction can be made for any category $\catE$ with finite limits.

\medskip\noindent
If $\catE$ is lextensive, the coproduct + in $\catE$ gives direct sum in $\Spn (\catE)$ :
\begin{align*}
\Spn (\catE) (X,Y + Z) &\cong \Spn (\catE) (X,Y) \times \ \Spn  (\catE) (X,Z)
\\
\Spn (\catE) (X + Y, Z) &\cong \Spn (\catE) (X,Z) \times \ \Spn  (\catE) (Y,Z)\,.
\end{align*}
This implies that Spn$(\catE)$ is a category with homs enriched in commutative monoids.   An enriched functor is one that preserves direct sums which exist.

\begin{Definition}[Lindner \cite{Lindner1976}]
A \emph{Mackey functor on} $\catE$ is an enriched functor
 \[
 M : \Spn (\catE) \longto \Mod _k\,.
 \]
\end{Definition}

\noindent
The Dress definition is recaptured when $\catE$ is finite $G$-sets, and $M^\ast$ and $M_\ast$ are obtained by composing $M$ with
\[ \xymatrix{
\catE^{\op} \ar[rr]^-{(\ )^\ast} && \Spn (\catE) && \catE \ar[ll]_-{(\ )_\ast}
}\]
where 
\begin{align*}
(f : X \longto Y)^{\ast} &=  (Y \stackrel{f}{ \longleftarrow} X \stackrel{1_X}{ \longto} X) 
\\
(f : X \longto Y)_{\ast} &= (X \stackrel{1_X}{ \longleftarrow} X \stackrel{f}{ \longto} Y) \,.
\end{align*}

\subsection{The category of Mackey functors\label{Lecture3sub2}}

The category of Mackey functors on $\catE$ is the commutative-monoid-enriched functor  category
\[
\Mky _k (\catE) : = [\Spn  (\catE), \Mod _k]_{\add}\,.
\]
It is closed monoidal by Day convolution:
\begin{align*}
(M \ast N) (Z) &= \int^Y M(Z\times Y) \otimes_k NY
\\
N^M (Z) &= \Mky _k (\catE) (M(Z \times -),N)
\end{align*}
the unit $J$ is the \emph{Burnside functor}. That is,
$JX$ is the free $k$-module on the commutative monoid under coproduct of isomorphism classes of objects of $\catE(X)$. 

\begin{Proposition}
 Green functors $A$ are monoids in $\Mky _k (\catE)$.
\end{Proposition}
\noindent
(The algebra structure in the Green definition gives $AX \otimes_k AY \stackrel{\mu}{\longto} A (X \times Y)$, 
$k \stackrel{\eta}{\longto} A\,1$  corresponding to $A \ast A \stackrel{\mu}{\longto}  A \stackrel{\eta}{\longleftarrow}  J)$.

\begin{Definition}[M. Barr \cite{Barr1979}]
A monoidal category $\funcM$ is said to be $\ast$-\emph{autonomous} when there exists an equivalence of categories
 $S : \funcM^{\op} \longto \funcM$
 and a natural isomorphism $\funcM (B,SA) \cong \funcM \bigl(I,S(A\otimes B)\bigr)$.
\end{Definition}

\noindent 
Write $\fMky_k(\catE)$ for the full subcategory  of $\Mky_k(\catE)$ consisting of those $M$ 
with each $MX$ finitely generated and projective.
 
\begin{Theorem}[Panchadcharam--Street \cite{MFCCC}] {\upshape[}$G$ finite, $\catE$ = finite $G$-sets, $k$ a field\/{\upshape]}\\
$\fMky_k(\catE)$ is a monoidal full subcategory of $\Mky_k(\catE)$ which is $\ast$-autonomous 
and satisfies $S(M)X = (MX)^\ast$.
\end{Theorem}

\noindent
We also looked to some extent into the theory (Bouc~\cite{BoucBook}) of modules over Green functors $A$ 
including Morita equivalence for Green functors.

 \subsection{The monoidal centre\label{Lecture3sub3}}

A \emph{lax half braiding} on an object $A$ of a monoidal category $\funcM$ is a natural family of morphisms
\[
u_X : A \otimes X \longto X \otimes A
\]
such that
\[
\xymatrix{
 A \otimes X \otimes Y \ar[rr]^{u_X \otimes Y} \ar[rd]_{u_X \otimes 1_Y} 
& & X \otimes Y \otimes A \\
& X \otimes A \otimes Y \ar[ru]_{1_x \otimes u_x}}  
 \quad \mbox{and} \quad   
 \xymatrix{A \otimes I \ar[rr]^{u_I} \ar@{=}[rd] & & I \otimes A \ar@{=}[ld] \\
 & A }
\]

The \emph{lax centre} $\catZ_\ell \funcM$ of $\funcM$ has objects pairs $(A,u)$ as above, with the obvious morphisms.   It is a monoidal category with tensor product
 \[
 (A,u) \otimes (B,v) = \bigl(A \otimes B, (u_- \otimes 1) \circ (1 \otimes v_-)\bigr)
 \]
and has a lax braiding
\[
\gamma^{}_{(A,u),(B,v)} = u_B : (A,u) \otimes (B,v) \to (B,v) \otimes (A,u)\,.
\]
A \emph{half braiding} is an invertible lax half braiding.  
The \emph{centre} $\catZ\funcM$ of $\funcM$ consists of the $(A,u)$ with $u$ invertible.

\medskip\noindent
If $\funcM$ is autonomous then
\[
\catZ_\ell \funcM = \catZ\funcM\,.
\]

\begin{example}   
The centre of the category of representations of a Hopf algebra $H$ 
is equivalent to the category of representations of the Drinfeld double $D(H)$ of $H$.
\end{example}   

\begin{example}   
The centre of the category $[G, \Set]$ (with cartesian monoidal structure) 
is equivalent to the braided monoidal category $[\Aut (G), \Set]$ as in Example~\ref{ProEx3} 
of Section~\ref{Lecture2}.
\end{example}   

\begin{Fact}   
A monoid $(A,u)$ in the monoidal lax centre $\catZ_\ell \funcM$ becomes a (lax) monoidal functor
\[
- \otimes A : \funcM \to \funcM
\]
via
\[ \xymatrix @C+7mm {
X \otimes A \otimes Y \otimes A \ar[r]^-{1\otimes U_Y \otimes 1}
& X \otimes Y \otimes A \otimes A \ar[r] ^-{1\otimes 1 \otimes U}
& X \otimes Y \otimes A \;.
}\]
\end{Fact}

This has a consequence for the \emph{Dress construction} 
which assigns to each Mackey functor $M$ on $\catE$ and each $Z \in \catE$ a Mackey functor
\[
M_Z := M \otimes (- \times Z) : \Spn  (\catE) \to \Mod_k\,.
\]
It implies, if $Z$ is a monoid in $\catZ_\ell \catE$ then each Green functor $A$ on $\catE$ defines a Green functor
\[ 
A_Z = \bigl( \xymatrix @C+2mm {
 \Spn(\catE) \ar[r]^-{-\times Z} & \Spn (\catE)\ar[r]^{A} & \Mod_k} \bigr)\,.
\]

\begin{GenQuest}   
If $\funcM$ belongs to a particular class of monoidal categories, when does $\catZ_\ell \funcM$ (or $\catZ\funcM$) belong to the same class?
\end{GenQuest}   

For monoidal functor categories, there are results along these lines in various papers with Day, Panchadcharam, Pastro and the author. A paper~\cite{PastroSt2008} with Pastro was motivated by Tambara~\cite{Tambara2006}. 
Now we shall look at some of that work.
 
Suppose $\catA$ is a monoidal $\catV$-category. 
We are ultimately interested in the (lax) centre of $[\catA,\catV\,]$. 
However, we first look at the functor category $[\catA^{\op}  \otimes \catA , \catV\,]$ 
whose objects might be called $\catV$-modules from $\catA$ to itself.

\medskip 
A $\catV$-functor $T : \catA^{\op}  \otimes \catA \to  \catV$ is a \emph{left Tambara module} 
when it is equipped with a natural family of morphisms
 \[
 \alpha^{}_\ell (A,X,Y) : T(X,Y) \to T (A\otimes X, A \otimes Y)
 \]
such that $\alpha^{}_\ell (I,X,Y) = 1_{T(X,Y)}$ and 
\[ \xymatrix{
T(X,Y) \ar[rd]_{\alpha^{}_\ell (A\otimes A',X,Y)\phantom{AAA}} \ar[rr]^{\alpha^{}_\ell (A',X,Y)}  
 && T(A'\otimes X, A' \otimes Y) \ar[ld]^{\phantom{AAA}\alpha^{}_\ell (A,A'\otimes X, A' \otimes Y)} \;.
\\
&T(A\otimes A' \otimes X, A \otimes A' \otimes Y }
\]
Similarly for \emph{right Tambara module}.   
For a \emph{Tambara module} we have both, plus 
\[ \xymatrix{
 T(X,Y) \ar[rrrd]^{\alpha(A,B,X,Y)} \ar[d]_{\alpha^{}_r (B,X,Y)}   \ar[rrr]^{\alpha^{}_\ell (A,X,Y)} 
  &&& T(A\otimes X,A\otimes Y) \ar[d]^{\alpha^{}_r (B,A\otimes X,A\otimes Y)}
\\
 T(X\otimes B,Y\otimes B)  \ar[rrr]_{\alpha^{}_\ell (A,X\otimes B, Y\otimes B)} 
  &&& T(A\otimes X\otimes B, A \otimes Y \otimes B) \;.
}\]

\bigskip\noindent 
If $\catA$ is closed, the $\alpha^{}_\ell, \alpha^{}_r, \alpha$ correspond to natural families
\begin{align*}
\beta_\ell (A,X,Y) &: T(X,Y^A) \longto T (A \otimes X,Y) \qquad
\\
\beta_r (B,X,Y) &: T(X,^{\; B}\!Y) \longto T (X \otimes B,Y) \qquad 
\\
\beta (A,B,X,Y) &: T(X,^{\; B}\!Y^A) \longto T (A \otimes X \otimes B,Y)\,.
\end{align*}
We call the system consisting of
\[ \begin{cases}
[\catA,\catV\,]  \longto  [\catA^{\op} \otimes \catA,\catV\,]
\\
F  \longmapsto T_F 
\\
T_F (X,Y) = F(Y^X)  
\end{cases}
\]
a \emph{Cayley} $\catV$-\emph{functor}: 
it takes the Cauchy tensor product (convolution) to ``tensor product over $\catA\,$'' 
(= composition of endomodules of $\catA$).   
The image of the Cayley $\catV$-functor consists of the left Tambara modules for which the $\beta_\ell$ are invertible.

A Tambara module is \emph{left strong} when the $\beta_\ell$ are invertible.   
It is \emph{strong} when the $\beta$ are invertible. 
We use the subscript `$\elless$' for left strong and `$\ess$' for strong.

\begin{Proposition}
If $\catA$ is closed monoidal, then we have that
$\catZ_\ell [\catA,\catV\,] \simeq \TambLS (\catA)$ and  
$\catZ [\catA,\catV\,]        \simeq \TambS (\catA)$.
Moreover, if $\catA$ is autonomous, every Tambara module is strong.
\end{Proposition}

\begin{Definition}[Pastro--Street \cite{PastroSt2008}]
The \emph{double} of a monoidal $\catV$-category $\catA$ is the $\catV$-category $\catD\catA$ 
with the same objects as $\catA^{\op} \otimes \catA$ and
\[
\catD\catA \bigl((X,Y),(U,V)\bigr) = \int^{A,B} \catA (U,A \otimes X \otimes B) \otimes \catA (A\otimes Y \otimes B,V)\,.
\]
\end{Definition}

\begin{Proposition}
$\Tamb(\catA) \simeq [\catD\catA,\catV\,]$.
\end{Proposition}

If $\catA$ is closed, there are canonical morphisms
\begin{align*}
&\tilde \beta_\ell : (X,Y^A) \longto (A \otimes X,Y) \\
&\tilde \beta : (X,^{\; B}\!Y^A) \longto  (A \otimes X \otimes B,Y)\,.
\end{align*}
in $\catD\catA$.   Inverting these gives the categories of fractions $\catD_{\elless}\catA$ and $\catD_{\ess}\catA$, respectively.

\begin{Proposition}
If $\catA$ is right closed, then
$\TambLS (\catA) \simeq [\catD_{\elless} \catA,\catV\,]$.
\\
If $\catA$ is closed, then
$ \TambS (\catA) \simeq [\catD_{\ess} \catA,\catV\,]$.
\end{Proposition}
\noindent
In these cases, we have
\[
\catZ_\ell [\catA,\catV\,] \simeq [\catD_{\elless} \catA,\catV\,] 
\quad \text{ and } \quad
\catZ [\catA,\catV\,] \simeq [\catD_{\ess} \catA,\catV\,]\,.
\]

If $\catA$ is autonomous, we have
\[
\catZ [\catA,\catV\,] \simeq \Tamb  (\catA) \simeq [\catD\catA,\catV\,].
\]
In particular, this applies when $\catA = \Spn \catE$ to yield:
\begin{Proposition}
$\catZ \Mky _k (\catE) \simeq [\catD \Spn \catE, \Mod_k]_{\add}\,$.
\end{Proposition}

\noindent
Tambara~\cite{Tambara2008} goes further in the case where $\catE$ is the category $G$-set 
of finite $G$-sets for a finite group $G$. 
He shows that $\Mky_k (\GSet)$ has its centre of the form $\Mky_k (\catE)$ for some $\catE$.

Here is a conjecture which gives Tambara's result as a special case. 
Let $\catC$ be a small category such that its finite coproduct completion $\ffam \catC$ is lextensive.  
Let $\Aut \catC$ be the category of automorphisms in $\catC$; 
that is, an object is a pair $(X,u)$ where $u : X \to X$ is invertible in $\catC$.

\begin{Conjecture}
$\catZ \Mky _k (\ffam  \catC) \simeq \Mky _k (\ffam\Aut \catC)$   
\end{Conjecture}

\noindent
Tambara has a long proof of this for $\catC$ the category of connected $G$-sets for a finite group $G$.

 \subsection{Classifying spaces\label{Lecture3sub4}}

Let $G$ be a compact connected Lie group.   Let $\classB G$ denote the classifying space of $G$ with universal $G$-fibre bundle $\univE G \to \classB G$: every $G$-fibre bundle $p : E \to X$  appears in a pullback square as follows.
\[
\xymatrix{
    E \ar[r]  \ar[d]_p  & \univE G  \ar[d]^p \\
   X \ar[r]_f    & \classB G 
  }
\]
For some classic $G$ there are nice explicit models of $\classB G$. 
A major problem is to calculate the cohomology groups of $\classB G$.

 \begin{Theorem}[Borel \cite{Borel1967}]
Let $\normN T$  denote the normalizer of a maximal torus $T$ of $G$ 
and let $\Weyl_G = \normN T/T$ be the Weyl group.   
If the prime $p$ does not divide $\# \Weyl_G$  then
 \[
 \classB \normN T \longto \classB G
 \]
induces an isomorphism on $\mmod p$ cohomology.
\end{Theorem}

\begin{Theorem}[Dwyer--Miller--Wilkerson \cite{DMW1987}]
Suppose $ G = \SO(3)$ and $p = 2$ with $\#\Weyl_G $ even. 
Then the square
\[
\xymatrix{
    \classB \LieD_8 \ar[r]  \ar[d]   & \classB  \OO_{24}  \ar[d] \\
   \classB \OO(2) \ar[r]     & \classB \SO(3) 
  }
\]
is seen by $\mmod 2$ cohomology as a homotopy pushout.
\end{Theorem}
 
Now take any prime number $p$. Let $\catA_p (G)$ be the category of nontrivial elementary abelian $p$-subgroups of $G$ 
(there are only finitely many up to conjugacy) where the morphisms are restrictions of inner automorphisms of $G$.   
Let $C(E)$ be the centralizer of $E \in \catA_p (G)$ in $G$. In fact,  $\BC(E)$ is homotopy equivalent to a functor in the variable $E \in \catA_p (G)$.
 
\begin{Theorem}[Jackowski--McClure \cite{JackMcC1992}]
The map
 \[
 \hocolim_{E \in \catA_p(G)^{\op}} \BC(E) \longto \classB G
 \]
induces an isomorphism on $\mmod p$ cohomology.
 \end{Theorem}
\begin{proof}[Outline of proof:] 
\begin{step} There is a standard spectral sequence associated with the cohomology of a homotopy colimit.   
As usual, to actually use a spectral sequence we need some collapsing.   
This involves proving:
 \begin{enumerate}[(a)]
\item 
(acyclicity) \quad $\lim_{E \in {\catA_{p}(G)}^{\op}}^{i} \coHom^{j}\bigl(\BC(E)\bigr) = 0$ for all $j$ and all $i>0$.
\item 
$ \coHom^\ast (\classB G) \cong \lim_{E \in \catA_p (G)^{\op}} \coHom^\ast \bigl(\BC(E)\bigr)$.
\end{enumerate} 
Here $\lim^i$ is the \emph{$i$-th derived functor} of
 \[
 \lim : [\catA_p (G) ,\Ab \,] \longto\Ab \,.
 \]
 For a functor $M : \catB^{\op}  \to\Ab  $, we have
 \[
 \lim M = [\catB^{\op}  ,\Ab \,] (\Delta \ZZ , M)
 \]
(where $\Delta \ZZ$ is the constant functor at $\ZZ$) so the $i$-th derived functor of $\lim$ is
 \[
 \coHom^i (\catB;M) := \Ext ^i_{[\catB^{\op},\Ab \,]} (\Delta \ZZ , M)\,.
 \]
 Call $M$ \emph{acyclic} when $ \coHom^i (\catB,M)=0$ for all $i>0$.
\end{step}

\begin{step}
A functor $M : \catB^{\op} \to\Ab  $ is said to be \emph{proto-Mackey} 
when its extension $M_+ : (\ffam \catB)^{\op} \to\Ab  $ extends further to a Mackey functor
 \[
 \MM : \Spn  \ffam \catB \longto\Ab \,.
 \]
\end{step}
 
\begin {Observation}
 $ \coHom^i (\catB;M) \cong  \coHom^i (\ffam  \catB; M_+)$\\ 
It is proved in~\cite{JackMcC1992} that $\ffam \catA_p(G)$ is lextensive enough for Mackey functors to be useful and that
 \[
 M(E) = \coHom^j (\BC(E);\ZZ /p)
 \]
is a proto-Mackey functor on $\catA_p (G)$.
\end {Observation}
For this $M$, it remains to prove:
\begin{inparaenum}[(a)]
\item 
$M$ is acyclic, and
\item 
$\coHom^j (\classB G ; \ZZ /p) \cong \lim M$.
\end{inparaenum}
The proof of (b) imitates the induction theorem of Dress~\cite{Dress1973} in Mackey functor theory.

\begin{step}
The proof of (a) is achieved by showing 
\begin{enumerate}[(a$'$)] 
\item \itshape 
every proto-Mackey functor
 $ \catA_p (G)^{\op}  \longto \Mod_{\ZZ /p} $
is acyclic.
\end{enumerate}

\noindent
Two simple lemmas below inspire the proof of (a$'$).
\end{step}
\end{proof}
\begin{Lemma}
If $\catC$ has a terminal object, every functor $M : \catC^{\op}  \longto\Ab  $ is acyclic.
\end{Lemma}

\begin{proof}
$\lim M = M1$ is exact in $M$ and so the derived functors of $\lim$ vanish. 
\end{proof}

\begin{Lemma}
Suppose  $\Gamma$  is a finite group and $M$ is a $\Gamma$-module. 
If multiplication by~$\# \Gamma$ as a function $M \to M$ is invertible, then $\coHom^i (\Gamma;M)=0$ for all $i >0$.
\end{Lemma}
 
\begin{proof}
$\Delta\ZZ  \in \xymatrix{
 [\Gamma^{\op},\Ab\,] \ar@<1ex>[rr]^-{U} &&\Ab   \ar@<1ex>[ll]_-{\perp}^-{\Set (\Gamma,-)} 
} \ni N$.
\[
 N \cong \Ab(\ZZ ,N) \cong [\Gamma^{\op},\Ab \,] (\Delta \ZZ , \Set ) (\Gamma,N) 
 \text{ implies } \coHom^i (\Gamma; \Set (\Gamma,N)) \cong 0 \text{ for } i>0\,.
 \]  
But $M$ is a retract of $\Set (\Gamma,M)\,$:
$
 \xymatrix{
  M  \ar[r]    \ar@/_2pc/[rr]_{\text{mult. by $\# \Gamma$}} &\Set  (\Gamma,M) \ar[r] & M 
   } 
$. 
\end{proof}

\section{Monoidal category theory for manifold invariants\label{Lecture4}}

\subsection{Preliminary remarks}

In Section~\ref{Lecture1}, I outlined how monoidal categories are relevant to braids and links.  
Work appearing in the period 1985--90 by Jones, Yetter, Jimbo, Drinfeld, Witten and Atiyah 
connected quantum field theories, low-dimensional topology and categories. 
This led, in the next few years, to invariants for 3-manifolds $M$ where the input was a special kind of monoidal category.
 
The Reshetikhin--Turaev invariant $\RT_\catB (M)$ constructed in~\cite{ReshTur1991} is based on a \emph{modular category} $\catB$ : that is, $\catB$ is a semisimple $\Vect $-enriched tortile monoidal category, with finitely many simple objects, and subject to a non-degeneracy condition.  One presents the manifold $M$ by surgery along a framed link and the link is coloured by simple objects of $\catB$.  

The Turaev--Viro invariant $\TV_\catC(M)$ constructed in~\cite{TuraevYa1992} and~\cite{BarWest1996} is based on a \emph{spherical category} $\catC$: that is, a pivotal monoidal category for which the two naturally defined traces agree.   
Also $\catC$ must be semisimple $\Vect$-enriched with finitely many simple objects, and subject to invertibility of dim $\catC$. 
The important difference is that $\catC$ need not be braided.   
One presents $M$ by a triangulation, colours the edges by simple objects of $\catC$ and evaluates coloured tetrahedra according to the $6j$-symbols of $\catC$. 
 
 Modular implies spherical, and it was shown that
 \[
 \TV_\catB (M) = \RT_\catB (M) \RT_\catB (-M)\,.
 \]
M\"uger~\cite{Muger2003a},~\cite{Muger2003b} proved that such spherical $\catC$ have their centres $\catZ(\catC)$ modular.   
Turaev conjectured generally that
 \[
 \TV_\catC (M) = \RT_{\catZ(\catC)} (M)
 \]
and some special cases were proved by Kawahigashi--Sato--Wakui~\cite{KSW2005}. 
The conjecture is proved by Turaev--Virelizier~\cite{TurVirez} using ideas of Brugui\'eres--Virelizier~\cite{BrugVirel2008} which stemmed from a generalization of the $\RT$-invariant due to Lyubashenko~\cite{Lyub1995}. 
Lyubashenko only requires a tortile monoidal category $\catB$ for which the ``cocentre" 
 \[
 C = \int^{x\in\catB} X \otimes X^\ast
 \]
 exists. Then $C$ is a Hopf monoid in $\catB$ equipped with a Hopf pairing $\omega : C \otimes C \to I$.
 
The problem comes in trying to apply this in the case $\catB = \catZ(\catC)$ for suitable spherical $\catC$.   
It is not known how to obtain the simple objects in $\catZ(\catC)$ from those in $\catC$, for example.    
What is needed is a way of calculating the cocentre of $\catB = \catZ(\catC)$.

 If $\catC$ is braided and has a cocentre $A$ then $\catZ(\catC)$ is the category of $A$-modules.   
 The Hopf monoid structure on $A$ provides the endofunctor
 \[
 A\otimes -: \catC \longto \catC
 \]
 with the structure of opmonoidal monad.   
 The notion of opmonoidal monad makes sense on any monoidal category $\catC$, braided or not; 
 whereas bimonoid in $\catC$ requires a braiding (or the like).
 
This led Brugui\`eres--Virelizier~\cite{BrugVirel2007} to develop the ideas of Moerdijk~\cite{Moerdijk2002} concerning opmonoidal monads. 
They developed the notion of Hopf monad on an autonomous monoidal category in the sense of an opmonoidal monad with antipode.   
Then, with Lack~\cite{BrugVirelLack}, this was extended to arbitrary monoidal categories. 
 
In a different direction, the monoidal centre of a modular category was also studied by Fr\"olich--Fuchs--Runkel--Schweigert~\cite{FFRS2006} and Kong--Runkel~\cite{KongRunkel}. 
It was realized that Morita equivalence classes of monoids in the chiral modular category of a Rational Conformal Field Theory correspond to consistent sets of boundary conditions. 
Certain commutative monoids in the monoidal centre describe the RCFT in ``full''.   
A construction was produced which assigned to a monoid $A$ in a modular category $\catB$, a commutative monoid $zA$ in the monoidal centre $\catZ(\catB)$ of $\catB$; 
they call $zA$ the \emph{full centre} of $A$. 
Recently, Davydov~\cite{Davydov2010} generalized the notion of full centre to any monoidal category $\catB$ by a characterizing universal property.
 
In studying the Brugui\`eres--Virelizier work, I was led to look at examples of ``2-monoidal categories'' 
in the sense of Aguiar--Mahajan~\cite{AguiarMahajan}. 
These are pseudomonoids (``monoidales'') in the monoidal 2-category of monoidal categories and monoidal functors.   
I call them \emph{duoidal categories}.  
Recall from~\cite{BTC} that a pseudomonoid in the monoidal 2-category of monoidal categories and \emph{strong} monoidal functors is essentially a braided monoidal category: 
by an Eckmann--Hilton argument the two tensor products are forced to be isomorphic and a braiding appears.

My treatment here of the category theory inspired by the 3-manifold invariants is joint work with Brian Day.
In fact, his paper~\cite{DayMid4} is also relevant to duoidal categories. 
 
In studying the full centre of Davydov, I was led to pseudomonoids in a different monoidal 2-category.  
I am grateful to Michael Batanin for further developments on  this topic.   
Coincidentally also, Batanin's work~\cite{Batanin2011a},~\cite{BatMarkl} with Martin Markl has led to duoidal categories. 
Finally, I might point to a paper~\cite{BookSt} which has arisen since these lectures were delivered 
and continues the study of duoidal categories in the same spirit.

\subsection{Duoidal categories}
 
Special cases of duoidal category have appeared in the literature. 
The general notion is called  \emph{2-monoidal category} by Aguiar--Mahajan~\cite{AguiarMahajan}. 
It also has importance in the work of Batanin--Markl~\cite{Batanin2011a},~\cite{BatMarkl}. 
 
Suppose $\catV$ is a cosmos in the sense of B\'enabou: 
that is, a symmetric monoidal closed complete and cocomplete category.   
A \emph{duoidal} $\catV$-\emph{category} is a $\catV$-category $\catF$ equipped with two monoidal $\catV$-category structures:
 \[
 \ast  : \catF \otimes \catF \longto \catF \text{,  } J \in \catF
 \quad\text{and} \quad
  \circ : \catF \otimes \catF \longto \catF  \text{,  } \catI \in \catF\,,
 \]
 such that the second structure is monoidal with respect to the first.
 Explicitly, $\circ : \catF \otimes \catF \longto \catF$ and $\catI  : \scrI \longto \catF$ are monoidal functors 
 where $\catF$ is monoidal via $\ast$ and $J$. 
 Apart from the two lots of associativity and unit constraints, the extra structure involves a middle-of-four-interchange morphism
 \[
 \gamma : (A \circ B) \ast (C \circ D) \longto (A \ast C) \circ (B \ast D)
 \]
 and morphisms
\[
\catI \ast \catI \stackrel{\mu}{\longto} \catI \stackrel{\tau}{\longto} J \stackrel{\delta}{\longto} J \circ J\,.
\]
In particular, $\catI$ is a monoid in $(\catF,\ast,J)$ while $J$ is a comonoid in $(\catF,\circ,\catI)$.

\subsection{Examples of duoidal categories}

\begin{example}\label{DuoEx1}%
$\catF$ = Rep, $\ast = \bullet$ (Cauchy product), $\circ = \underline{\times}$ (Hadamard product).
\end{example} 

\begin{example}\label{DuoEx2}%
$(\catA, \otimes)$ any monoidal category  with finite coproducts: $\catF = \catA$, $\ast = +$, $\circ = \otimes$.
\end{example} 

\begin{example}\label{DuoEx3}%
If $(\catA, \otimes)$ is a braided monoidal category: $\catF = \catA$, $\ast = \otimes$, $\circ = \otimes$.
\end{example} 

\begin{example}\label{DuoEx4}%
If $(\catF,\ast,\circ)$ is duoidal then so are $(\catF^{\op},\circ,\ast)$, $(\catF,\ast,\circ^{\rev})$ and $(\catF,\ast^{\rev},\circ)$.
\end{example} 

\begin{example}[Day--Street] \label{DuoEx5}%
If $\catC$ is any monoidal $\catV$-category: $\catF = [\,\catC^{\op} \otimes \catC,\catV\,]$, 
$\ast$ is convolution using the monoidal structure on $\catC^{\op} \otimes \catC$, 
and $\circ$ is ``tensor product over $\catC\,$''.   
More explicitly, 
\begin{align*}
\catF (M \ast N,L) &\cong \int_{U,V,X,Y} \catV \bigl(M(U,X) \otimes N(V,Y),L(U\otimes V,X\otimes Y)\bigr)\,,
\\
\catF (M\circ N, L) &\cong \int_{A,B,Z} \catV \bigl(N(Z,B) \otimes M(A,Z),L(A,B)\bigr)\,.
\end{align*}
Also
\[
J(A,B) = \catC (A,I) \otimes \catC (I,B) \quad\text{and} \quad
 \catI (A,B) = \catC (A,B)\,.
\]
The isomorphisms defining $M\ast N$ and $M\circ N$ are induced by morphisms
\begin{align*}
\xi &: M(U,X) \otimes N(V,Y) \longto (M \ast N) (U\otimes V, X \otimes Y)
\\
\zeta &: N(Z,B) \otimes M(A,Z) \longto (M\circ N) (A,B)\,.
\end{align*}
The middle-of-four morphism
\[
\gamma : (M \circ N) \ast (M' \circ N') \longto (M \ast M') \circ (N\ast N')
\]
is induced by the composite
\begin{align*}
\lefteqn{N(H,X)\otimes M(U,H) \otimes N'(K,Y) \otimes M'(V,K)}
\\
&&& \xymatrix{\ar[r]^-{\cong}&}
N(H,X) \otimes N'(K,Y) \otimes M(U,H) \otimes M'(V,K) 
\\
&&&\xymatrix{\ar[r]^-{\xi \otimes \xi}&}
(N\ast N') (H \otimes K, X \otimes Y) \otimes (M\ast M') (U \otimes V,H \otimes K) 
\\
&&&\xymatrix{\ar[r]^-{\zeta}&} 
 \bigl((M\ast M')\circ (N \ast N')\bigr) (U \otimes V, X \otimes Y)\,.
\end{align*}
\end{example} 

\begin{example}[Batanin--Markl \cite{BatMarkl}]\label{DuoEx6}%
Let $\catC$ be a category.  
Eilenberg and I used the term \emph{derivation scheme} on $\catC$ for a function $\XX$ which assigns, to each pair of parallel morphisms 
\[ \xymatrix{
  a \ar@/^1pc/[rr]^\alpha  \ar@/_1pc/[rr]_{\alpha'}  \ar@{}[rr] && b 
 }\]
in $\catC$, a set $\XX_{(\alpha,\alpha')}$.  The elements of this last set are depicted as 2-cells
\[ \xymatrix{
 a \ar@/^1pc/[rr]^\alpha  \ar@/_1pc/[rr]_{\alpha'}  \ar@{}[rr]|{\Downarrow \; x}&& b 
 }\]
Write $\Ds(\catC)$ for the category of derivation schemes on $\catC$.  
Then $\catF = \Ds(\catC)$ is duoidal with
\begin{align*}
(\XX \ast \YY)_{(\alpha,\alpha')}  &=  \biggl\{  
 \xymatrix{a \ar@/^1pc/[rr]^\beta  \ar@/_1pc/[rr]_{\beta'}  \ar@{}[rr]|{\Downarrow \; x}&& c  \ar@/^1pc/[rr]^\gamma 
\ar@/_1pc/[rr]_{\gamma'} \ar@{}[rr]|{\Downarrow \; y}&& b}  
\biggm | \begin{array}{lll}
x \in \XX ,  &  y \in \YY,    \\
\alpha = \gamma \beta ,  &  \alpha' = \gamma' \beta'   
\end{array}
\biggr\}\\
J_{(\alpha,\alpha')} &= \begin{cases}
               1 & \text{when } \alpha = \alpha' = 1_a \\
\emptyset & \text{otherwise}
\end{cases} \\
(\XX \circ \YY)_{(\alpha,\alpha'')}  &=  \biggl\{
 \xymatrix{\ar[rr]|{\alpha'} a \ar@/^1.5pc/[rr]^{\alpha}_{\Downarrow \; x}  \ar@/_1.5pc/[rr]_{\alpha''}^{\Downarrow \; y} & & b 
 }  
\biggm|
  x \in \XX ,    y \in \YY 
\biggr\} \\
\catI_{(\alpha,\alpha'')} & = \begin{cases}
                1 & \text{when } \alpha = \alpha' \\
\emptyset & \text{otherwise}.
\end{cases}
\end{align*}
The morphism
\[
\gamma : (\XX \circ \XX') \ast (\YY \circ \XX') \longto (\XX \ast \YY) \circ (\XX' \ast \YY')
\]
is defined by
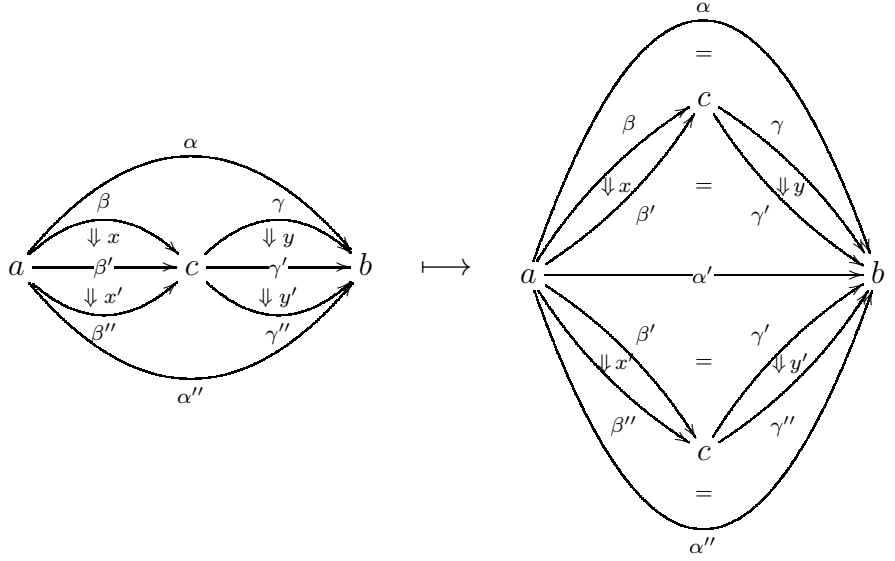
\[
  \xymatrix{
   \ar@/_3.5pc/[rrrr]_{\alpha''}  \ar@/^3.5pc/[rrrr]^{\alpha}\ar[rr]|{\beta'} 
   a   
      \ar@/^1.5pc/[rr]^{\beta}_{\Downarrow \; x}  \ar@/_1.5pc/[rr]_{\beta''}^{\Downarrow \; x'} 
   & & c 
      \ar[rr]|{\gamma'}   \ar@/^1.5pc/[rr]^{\gamma}_{\Downarrow \; y}  \ar@/_1.5pc/[rr]_{\gamma''}^{\Downarrow \; y'} 
   && b
 }  
 \quad \longmapsto \quad
\raise3.5cm\hbox{\xymatrix{
 &&\ar@{}[d]|{=} \\
 &&c\ar@{}[dd]|{=}\ar@/^0.5pc/[rrdd]^(.3){\gamma}  
  \ar@/_0.5pc/[rrdd]_{\gamma'} \ar@{}[rrdd]|{\Downarrow\, y}
 \\
 \\
 a
  \ar@/_8pc/[rrrr]_{\alpha''}  \ar@/^8pc/[rrrr]^\alpha  \ar@/^0.5pc/[rruu]^(.7)\beta 
  \ar@/_0.5pc/[rruu]_{\beta'} \ar@{}[rruu]|{\Downarrow\, x} \ar[rrrr]|{\alpha'} \ar@/^0.5pc/[rrdd]^{\beta'}  
  \ar@/_0.5pc/[rrdd]_(.7){\beta''}   \ar@{}[rrdd]|{\Downarrow\, x'} 
 &&& & b \rlap{\;.}
 \\ \\
 &&c\ar@{}[uu]|{=} \ar@/^0.5pc/[rruu]^{\gamma'}   \ar@/_0.5pc/[rruu]_(.3){\gamma''} \ar@{}[rruu]|{\Downarrow\, y'}
 \\
 &&\ar@{}[u]|{=}
 }}
\]
\end{example}

\subsection{Duoidal and bimonoidal functors}

Let $\Mon(\VCat)$ denote the monoidal 2-category whose objects are monoidal $\catV$-categories and whose morphisms are monoidal $\catV$-functors.   A duoidal $\catV$-category is precisely a pseudomonoid (monoidale) in $\Mon(\VCat)$.

\bigskip

A \emph{duoidal} $\catV$-\emph{functor} $\Phi : \catG \longto \catF$ is precisely a monoidal morphism in $\Mon(\VCat)$.  
It means that $\Phi$ is both a monoidal $\catV$-functor $(\catG,\ast) \longto (\catF,\ast)$ and a monoidal $\catV$-functor $(\catG,\circ) \longto (\catF,\circ)$ which are compatible in a straightforward way.   
As suggested by the Batanin--Markl example, we go ahead and call $\ast$ the \emph{horizontal} monoidal structure on $\catF$ and $\circ$ the \emph{vertical}.

\bigskip

A \emph{duoid} in $\catF$ is a duoidal $\catV$-functor $A  : \scrI \longto \catF$.   
That is, $A$ is both a horizontal and vertical monoid
\[
A \ast A \stackrel{\mu_h}{\longto} A\,, 
\quad J  \stackrel{\eta_h}{\longto} A\,, 
\quad A \circ A  \stackrel{\mu_v}{\longto} A\,,
\quad \catI \stackrel{\eta_v}{\longto} A
\]
in $\catF$ such that $\mu_v$ and $\eta_v$ are horizontal monoid morphisms.

\bigskip

A \emph{bimonoidal} ($\catV$-)\emph{functor} $\Theta : \catG \longto \catF$ is precisely an opmonoidal  morphism in $\Mon(\VCat)$.   It means that $\Theta$ is horizontally monoidal, vertically opmonoidal and subject to compatibilities.

\bigskip

A \emph{bimonoid} in $\catF$ is a bimonoidal functor $\scrI \longto \catF$.  
That is, it is an object $A$ in~$\catF$ equipped with both a horizontal monoid and a vertical comonoid structure
\[
A \ast A \stackrel{\mu_h}{\longto} A\,, 
\quad J  \stackrel{\eta_h}{\longto} A\,, 
\quad A \stackrel{\delta_v}{\longto} A \circ A\,, 
\quad A \stackrel{\varepsilon_v}{\longto} \catI 
\]
such that $\delta_v$ and $\varepsilon_v$ are horizontal monoid morphisms.

In a braided monoidal category (Example~\ref{DuoEx3}), 
a bimonoid is a bimonoid (or bialgebra) in the usual sense while a duoid is a commutative monoid.

In the Batanin--Markl case (Example~\ref{DuoEx6}), a duoid is precisely a \emph{2-category}.

\subsection{Categories enriched in a duoidal category}

For any duoidal ($\catV$-)category $\catF$, there is a monoidal 2-category $\FCat$
whose objects are categories with homs enriched in the monoidal category $(\catF,\ast,J)$,
and the morphisms are enriched functors.  
The vertical structure of $\catF$ is used for the tensor product of $\FCat$: 
if $\catA$ and $\catB$ are horizontally enriched categories then $\catA \circ \catB$ is defined by
\begin{align*}
\ob \catA \circ \catB &= \ob\catA \times \ob\catB
\\
(\catA \circ \catB) \bigl((A,B),(A',B')\bigr) &= \catA (A,A')   \circ \catB (B,B')\,.
\end{align*}
Hence we can speak of \emph{monoidal} $\catF$-categories; 
they are monoidal objects $(\catA,\otimes,I)$ in  $\FCat$.   
Batanin has pointed out that $\catA(I,I)$ is then a duoid:
\[
\catA(I,I) \ast \catA(I,I) \stackrel{\mu_h}{\longto} \catA(I,I)
\]
is composition in $\catA$ while $\mu_v$ is the composite
\[
\catA (I,I) \circ \catA(I,I) = (\catA \circ \catA) \bigl((I,I),(I,I)\bigr)
 \stackrel{\otimes}{\longto} \catA (I \otimes I, I \otimes I) \cong \catA (I,I)\,.
\]
(This generalizes the fact that $\catA(I,I)$ is a commutative monoid in $\catF$ 
when $\catF$ is a braided monoidal category as in Example~\ref{DuoEx3}).

In some interesting cases the monoidal 2-category $\FCat$ is \emph{closed}: write $[\catA,\catB\,]$ for the internal hom.   
Then $[\catA,\catA\,]$ is a monoidal $\catF$-category for any $\catF$-category $\catA$.   
It follows that
\[
\zz(\catA) = [\catA,\catA\,] (1_\catA,1_\catA)
\]
is a duoid in $\catF$; we might call it the \emph{centre of} $\catA$.    
In particular, the centre of a (horizontal) monoid $M$ in $\catF$ is a duoid $\zz(M)$ in $\catF$.

\subsection{The duoidal category of endomodules of a monoidal\\category}

We now focus attention on Example~\ref{DuoEx5} of a duoidal category:
\[
\catF = [\,\catC^{\op} \otimes \catC,\catV\,] \qquad \qquad (\catC \text{ a small monoidal } \catV\text{-category})
\]
\begin{align*}
 (M \ast N) (A,B) &= \int^{U,V,X,Y} \catC (A,U \otimes V) \otimes \catC (X \otimes Y,B) \otimes M (U,X) \otimes N (V,Y)
\\
 (M \circ N) (A,B) &= \int^Z N(Z,B) \otimes M (A,Z)\,.
\end{align*}
There is a fully faithful functor
\begin{align*}
(\ )_\ast : [\,\catC,\catC\,] &\longto \catF
\\T &\longmapsto  T_\ast
\\T_\ast (U,V) &= \catC (U,TV)
\end{align*}
which is strong monoidal with respect to composition $\circ\,$.   
Each $T_\ast$ has a right dual $T^\ast$ for the vertical structure on $\catF$. 

\begin{Observations}
\begin{enumerate}[1.]
\item 
An endofunctor $G$ of $\catC$ is a monoidal comonad if and only if $G_\ast$ is a bimonoid in $\catF$.
\item 
An endofunctor $T$ is a monoidal monad if and only if $T_\ast$ is a duoid in $\catF$. 
\end{enumerate}
\end{Observations}

As a hint that we are not far removed from monoidal centres, 
notice that Tambara modules are precisely left $\catI$-, right $\catI$-bimodules in $(\catF,\ast)$.

A functor $\Phi : \catF^{\op} \longto \catF$ is defined by
\[
\Phi (F) (X,Y) = \int_{U,V} \catV \bigl(F(U,V),\catC(X\otimes U, V \otimes Y)\bigr)\,.
\]
It has a left adjoint $\Psi$ defined by
\[
\Psi (G) (U,V) = \int_{X,Y} \catV \bigl(G(X,Y),\catC(X\otimes U, V \otimes Y)\bigr)\,;
\]
the isomorphism
\[
\catF\bigl(G,\Phi(F)\bigr) \cong \catF\bigl(F,\Psi(G)\bigr)
\]
is immediate from symmetric monoidal closedness in $\catV$.

\begin{Theorem} 
$\Phi : (F^{\op}, \circ^{\rev}, \ast) \longto (\catF,\ast,\circ)$ is a duoidal $\catV$-functor. 
It is vertically normal{\upshape:} $\Phi (J) \cong \catI$.
\end{Theorem} 

\begin{proof} 
The identity morphisms in $\catC$ give
\[
I \longto \int_V \catC(V,V) \cong \Phi (\catI)(I,I)
\]
and hence $\eta_h : J \to \Phi (\catI)$ by Yoneda.
For $\mu_h : \Phi (F) \ast \Phi (G) \to \Phi (G\circ F)$, we require
\[
\Phi (F) (U,X) \otimes \Phi (G) (V,Y) \to \Phi (G\circ F) (U\otimes V, X \otimes Y)
\]
which amounts to having
\begin{multline*}
\lefteqn{\int_{H,K} \catV \bigl(F(H,K),\catC(U\otimes H,K\otimes X)\bigr) 
   \otimes \int_{R,S} \catV \bigl(G(R,S),\catC(V\otimes R, S \otimes Y)\bigr)}\\
 \longto \catV (F(Z,Q) \otimes G(P,Z)\,, \catC \bigl(U \otimes V \otimes P, Q \otimes X \otimes Y)\bigr)\,.
\end{multline*}
We use the projections where $H=Z$, $K=Q$ and $R=P$, $S=Z$ to move from the domain tensor product of two ends to 
\[
\catV (F(Z,Q), \catC \bigl(U \otimes Z, Q \otimes X)\bigr) \otimes \catV \bigl(G(P,Z),\catC(V\otimes P, Z \otimes Y) \bigr)
\]
and then, by tensoring, to
\[
\catV \bigl(F(Z,Q) \otimes G(P,Z), \catC (U\otimes Z, Q \otimes X) \otimes \catC (V \otimes P, Z \otimes Y) \bigr)\,.
\]
We get to the required codomain by composing with the composite 
\begin{align*}
\lefteqn{\catC(U\otimes Z, Q \otimes X) \otimes \catC (V \otimes P, Z \otimes Y)}
\\
 &&& \xymatrix @C+11mm { \ar[r]^{(-\otimes Y)\otimes (U\otimes -)} &}
 \catC (U \otimes Z \otimes Y, Q \otimes X \otimes Y) \otimes \catC (U \otimes V \otimes P, U \otimes Z \otimes Y)
\\
 &&&\xymatrix @C+4mm { \ar[r]^{\mathrm{comp.}}& } \catC (U\otimes V \otimes P, Q \otimes X \otimes Y)\,.
\end{align*}
For the vertical structure, the normality comes from the isomorphisms
\begin{align*}
 \catI (X,Y) &= \catC (X,Y) \cong \catC (X \otimes I, I \otimes Y)
 \\
   &\cong \int_{U,V} \catV (J(U,V),\catC(X\otimes U, V \otimes Y)) = \Phi (J) (X,Y)\,.
\end{align*}
For $\mu_v : \Phi (F) \circ \Phi (G) \to \Phi (F \ast G)$, we require a natural family
\[
\Phi (G) (W,V) \otimes \Phi (F) (U,W) \to \Phi (F \times G) (U,V).
\]
Using the calculation
\begin{align*}
 \Phi (F \ast G) (U,V) &= \int_{X,Y} \catV \bigl(F \ast G) (X,Y), \catC (U \otimes X, Y \otimes V)\bigr)
\\
 &\cong \int_{X,Y,H,K} \catV \bigl(F(X,H)\otimes G(Y,K),\catC (U\otimes X \otimes Y, H \otimes K \otimes V)\bigr)
\end{align*}
and a middle-of-four isomorphism, we see that we need
\[
\Phi (G) (W,V) \otimes G(Y,K) \otimes \Phi (F) (U,W) \otimes F (X,H) \longto \catC (U \otimes X \otimes Y, H \otimes K \otimes V)\,.
\]
Using the universal morphism
\[
\alpha^{}_F : \Phi (F) (W,V) \otimes F (X,H) \longto \catC (W \otimes X, H \otimes V)
\]
and the similar $\alpha^{}_G$, we have what we need if we have
\[
\catC (W \otimes Y,K \otimes V) \otimes \catC (U \otimes X, H \otimes W) 
 \longto \catC (U \otimes X \otimes Y, H \otimes K \otimes V);
\]
and, as before, we do have such, namely, 
\begin{align*}
\lefteqn{\catC(W\otimes Y, K \otimes V) \otimes \catC (U \otimes X, H \otimes W) }
\\
 &&& \xymatrix @C+14mm { \ar[r]^{(H\otimes-)\otimes (-\otimes Y)}&}
\catC (H \otimes W \otimes Y, H \otimes K \otimes V) \otimes \catC (U \otimes X \otimes Y, H \otimes W \otimes Y)
\\
 &&& \xymatrix @C+4mm { \ar[r]^{\mathrm{comp.}} & } \catC (U\otimes X \otimes Y, H \otimes K \otimes V)\,.
\end{align*}
This describes the duoidal structure on $\Phi$.   
The axioms require checking in detail.
\end{proof} 

\begin{Remark}
The convolution structure $\ast$ on $\catF$ is closed monoidal;
that is, we have that $\catF (M \ast N, L) \cong \catF (N, L^M)$, where
\[
L^M (X,Y) = \int_{U,V} \catV \bigl(M(U,V),L(U\otimes X,V\otimes Y)\bigr)\,.
\]
So
\[
\catI^F (X,Y) = \int_{U,V} \catV \bigl(F(U,V),\catC(U\otimes X, V \otimes Y)\bigr),
\]
which is very close to the formula for $\Phi (F) (X,Y)$.    
However the $X \otimes U$ replacing $U \otimes X$ is important for the consideration of centres.
\end{Remark}

If $\catC$ is left closed and suitably complete, we can define a ($\catV$-) functor as follows:
\[
\funcC : [\,\catC,\catC\,]^{\op} \longto [\,\catC,\catC\,]
\quad\text{whereby}\quad 
\mathop{\funcC(G)} Y = \int_V\laction{GV}{(V\otimes Y)}\,.
\]
Then
\begin{align*}
\Phi (G_\ast) (X,Y) &=
\int_{U,V} \catV \bigl(\catC(U,GV),\catC(X\otimes U, V \otimes Y)\bigr)
\\
&\cong \int_V \catC (X\otimes GV,V\otimes Y)
   \cong \catC \bigl(X,\funcC(G)Y\bigr)
\\
&\cong \funcC (G)_\ast (X,Y)\,.
\\ &\kern-20pt
\xymatrix{\ar @{} [dr] |{\cong}
[\,\catC,\catC\,]^{\op}   \ar[r]^{\funcC}\ar[d]_{(\ )_\ast } & [\,\catC,\catC\,] \ar[d]^{(\ )_\ast } 
\\
[\,\catC^{\op} \otimes \catC,\catV\,]^{\op} \ar[r]_\Phi \ar @{} [d]|*@{=}
 & [\,\catC^{\op} \otimes \catC,\catV\,] \ar@{} [d] |*@{=}
\\
*!D{\catF^{\op} }&*!D{\catF}
} \end{align*}

\begin{Lemma}
If $U$ has a left dual $\ldual{U}$ then
\[
U \otimes \funcC (G) Y \cong \int_V {}^{GV} (U \otimes V \otimes Y)\,.
\]
\end{Lemma}

\begin{proof} 
\begin{align*}
\catC \bigl(C,U\otimes \funcC (G)Y\bigr) &\cong \catC (^\wedge U \otimes X, \funcC (G) Y)
\\
&\cong \int_V \catC \bigl({}\ldual{U}\otimes X, {}^{GV}(V\otimes Y)\bigr)
  \cong \int_V \catC ({}\ldual{U} \otimes X \otimes GV,V\otimes Y)
\\
&\cong \int_V \catC (X\otimes GV,U\otimes V \otimes Y) \cong \catC \bigl(X,\int_V {}^{GV}(U\otimes V\otimes Y)\bigr)\,. 
\end{align*}
\end{proof}

\begin{Proposition}
If $\catC$ is left autonomous and suitably complete, then we have that for all endo-$\catV$-functors $F,G:\catC \longto \catC$,
\[
\mu_v : \Phi (F_\ast) \circ \Phi (G_\ast) \to \Phi (F_\ast \ast G_\ast)
\]
is invertible.
\end{Proposition}

\begin{proof} 
Using Yoneda and an earlier calculation, we have
\[
\Phi (F_\ast \ast G_\ast) (X,Y) \cong \int_{U,V} \catC (X \otimes FU \otimes GV, U\otimes V \otimes Y).
\]
However, using the Lemma, we have
\begin{align*}
 \bigl(\Phi(F_\ast)\circ \Phi(G_\ast)\bigr) (X,Y) &\cong \bigl(\funcC(F)_\ast \circ \funcC (G)_\ast\bigr)(X,Y)
\\
&= \catC \bigl(X,\funcC(F)(\funcC(G)Y)\bigr) 
  \cong \int_U \catC \bigl(X, {}^{FU}(U\otimes \funcC(G)Y)\bigr)
\\
&\cong 
\int_U \catC \bigl(X \otimes FU, U\otimes \funcC (G) Y\bigr) 
\\
& \cong \int_{U,V} \catC \bigl(X\otimes FU, {}^{GV}(U\otimes V \otimes Y)\bigr)
\\
&\cong \int_{U,V} \catC (X\otimes FU \otimes GV, U\otimes V \otimes Y)\,.  \qedhere
\end{align*}
\end{proof}

\begin{Corollary} 
Under the conditions of the Proposition, 
if $G$ is a monoidal comonad on $\catC$ then $\funcC (G)$ is a monoidal comonad on $\catC$.
\end{Corollary}

\begin{proof} 
We already noticed that $G$ is a monoidal comonad if and only if $G_\ast$ is a bimonoid in $\catF$.   
Then $\Phi(G_\ast)$ is a bimonoid in $\catF$ using
\begin{eqnarray*}
\mu_h       &:&   \Phi (G_\ast) \ast \Phi (G_\ast) \stackrel{\mu^{}_h}{\longto} \Phi (G_\ast\circ G_\ast)
  \xylongto{\Phi(\delta_v)} 
   \Phi (G_\ast)
\\
\mu_v &:& \Phi(G_\ast) \xylongto{\Phi(\mu^{}_h)} 
 \Phi (G_\ast \ast G_\ast) 
 \stackrel{\mu^{-1}_v}{\longto} \Phi (G_\ast) \circ \Phi (G_\ast)\,.
\end{eqnarray*}
However, $\Phi (G_\ast) \cong \funcC (G)_\ast\,$; so $\funcC (G)$ is a monoidal comonad.
\end{proof}

\subsection{Hopf bimonoids}

\begin{Definition} 
A bimonoid $A$ in any duoidal category $\catF$ is \emph{left Hopf} when the ``left fusion morphism''
\[
A \ast A \xylongto{1 \ast \delta} 
 (A \circ \catI) \ast (A \circ A) 
 \stackrel{\gamma}{\longto} (A\ast A) \circ (\catI \ast A) \xylongto{\mu \circ 1} 
  A \circ (\catI \ast A)
\]
is invertible.   
The bimonoid $A$ is \emph{right Hopf} when the ``right fusion morphism''
\[
A \ast A \xylongto{\delta \ast 1} 
 (A \circ A) \ast (A \circ \catI) \stackrel{\gamma}{\longto}
  (A \ast A) \circ (A \ast \catI) \xylongto{\mu \circ 1} 
  A \circ (A \ast \catI)
\]
is invertible. The bimonoid $A$ is \emph{Hopf} when it is both left and right Hopf.
\end{Definition} 

\begin{Proposition} 
Let $G$ be a monoidal comonad on $\catC$. 
In $\catF = [\,\catC^{\op} \otimes \catC,\catV\,]$, the bimonoid $G_\ast$ 
is left Hopf if the composite
\[
GX \otimes GY \xylongto{1 \otimes \delta} 
GX \otimes G^2 Y \xylongto{\mu^{}_{X,GY}} 
G (X \otimes GY)
\]
is invertible.   
In $\catF$, the bimonoid $G_\ast$ is right Hopf if the composite
\[
GX \otimes GY \xylongto{\delta \otimes 1} 
G^2 X \otimes GY \xylongto[10pt]{\mu^{}_{GX,Y}} 
G (GX \otimes Y)
\]
is invertible.
\end{Proposition} 

\noindent
(This links our Hopf notion with that of Brugui\`eres--Lack--Virelizier~\cite{BrugVirelLack} 
and Chikhl\-adze--Lack--Street~\cite{ChLaSt}).

\begin{proof} 
From the universal properties of $\ast$ and $\circ$, we easily calculate that
\begin{align*}
\catF (G_\ast \ast G_\ast,L) &\cong \int_{X,Y} L(GX \otimes GY, X \otimes Y)\,,
\\
\catF (G_\ast \circ (\catI \ast G_\ast),L) &\cong \int_{X,Y} L\bigl(G(X \otimes GY), X \otimes Y\bigr)\,,
\\
\catF (G_\ast \circ (G_\ast \ast \catI),L) &\cong \int_{X,Y} L\bigl(G(GX \otimes Y), X \otimes Y\bigr)
\end{align*}
and it remains to see that composites in the Proposition induce the fusion morphisms. 
\end{proof} 

\noindent
If the two composites in the Proposition are invertible, $G$ is called a \emph{Hopf comonad on} $\catC$. 

\begin{Remark} 
In general, I do not see that the converses hold in the Proposition.  We would require the morphisms
\[
\catC (X,Y) \otimes \catC (U,V) \stackrel{\otimes}{\longto} \catC (X \otimes U, Y \otimes V)
\]
to be strong monomorphisms in $\catV$.
\end{Remark} 

Suppose $G$ is a right Hopf comonad on a left autonomous monoidal category $\catC$.  
Suppose $\catC$ is complete enough and $G$ preserves limits.   
Then there is a distributive law
\[
\lambda : \funcC (G) G \longto G \funcC (G)
\]
between the comonads $G$ and $\funcC (G)$.  
To see this, note that the right Hopf condition gives a morphism
\[
I \stackrel{\eta}{\longto} GI \xymatrix@C+12pt { \ar[r]_{G (\text{unit})} & }
 G (GW \otimes \ldual{G}W) \cong GW \otimes G (\ldual{G}W)
\]
and hence a morphism
\[
\bar \lambda : \ldual{G}W \longto G(\ldual{G}W)\,.
\]
Then $\lambda$ is defined by the following diagram.
\[
\xymatrix @C+1.5cm {
 \int_V V \otimes GY \otimes \ldual{G}V \ar[r]^{\lambda_Y} \ar[d]^{\pr_{GW}} 
  &  \int_W G (W\otimes Y \otimes \ldual{G}W) \ar[d]^{\pr_W}
\\
GW \otimes GY \otimes \ldual{G}^2W \ar[d]_{1\otimes 1 \otimes {}^{\wedge}\delta_W} 
 & G(W\otimes Y \otimes \ldual{G}W) 
\\
GW\otimes GY \otimes \ldual{G}W \ar[r]_-{1\otimes 1 \otimes \bar \lambda}  
 &  GW \otimes GY \otimes G (\ldual{G}W) \ar[u]_{\mu_3} 
}
\]
The composite comonad (or wreath product)
\[
\Drin (G) = \mathop{\funcC (G)} G
\]
is called the \emph{Drinfeld double} of $G$ by Brugui\`eres--Virelizier~\cite{BrugVirel2008}.

\begin{Theorem} [Brugui\`eres--Virelizier]
If $\catC$ is autonomous, 
$G$ is a Hopf comonad on $\catC$, and $\funcC(G)$ exists 
then $\funcC(G)$ and $\Drin(G)$ are Hopf comonads on $\catC$, and there is an equivalence of monoidal categories
\[
\catC^{\Drin (G)}  \simeq Z (\catC^G)\,.
\]
{\upshape [}Here $\catC^G$ denotes the category of Eilenberg--Moore $G$-coalgebras.{\upshape ]}
\end{Theorem}

\begin{proof}[Outline of proof:] 
A $\Drin (G)$-coalgebra $ A\longto \int\limits_V V \otimes GA \otimes \ldual{G}V$ amounts to a natural family
\[
\alpha^{}_{A,V} : A \otimes GV \longto V \otimes GA
\]
satisfying two conditions. Then $A$ becomes a $G$-coalgebra via the coaction
\[
A \xylongto{1\otimes \eta} 
A \otimes GI \xylongto{\alpha^{}_{A,I}} 
 I \otimes GA \cong GA
\]
and lies in $Z(\catC^G)$ using
\[
A \otimes V \xylongto{1\otimes \delta} 
A \otimes GV \xylongto{\alpha^{}_{A,V}} 
V \otimes GA \xylongto{1\otimes \varepsilon_A} 
V \otimes A\,.
\]
Conversely, $\alpha^{}_{A,V}$ can be reconstructed from $\delta : A \longto GA$ and $u_W : A \otimes W \longto W \otimes A$ as 
\[
A\otimes GV \xylongto{u_{GV}} 
GV \otimes A \xylongto{1\otimes \delta} 
GV \otimes GA \xylongto{\varepsilon_V\otimes 1} 
 V \otimes GA\,.  \qedhere 
\]  
\end{proof}
 
\subsection{The full centre of a monoid}

We now turn to Davydov's idea~\cite{Davydov2010} about the full centre of a monoid in a monoidal category.   
First we remind you of the construction in~\cite{MCL}.

Let $\funcM$ be a braided monoidal bicategory (such as $\Cat$).  
For any pseudomonoid $A$ (that is, a monoidal category when $\funcM=\Cat$) there is a pseudo-cosimplicial category
\[ 
 \UseTips{}\entrymodifiers={+<4mm>!C}\xymatrix{ 
 \funcM(U,A) \ar @<6pt>[r]  \ar @<-6pt>[r]
& \funcM(U\otimes A,A) \ar[l]  \ar @<6pt>[r]  \ar[r]   \ar @<-6pt>[r]
& \funcM(U\otimes A\otimes A,A) }
\]
defined in a Hochschild sort of way.
 
The descent category is denoted by $\CP(U,A)$; the objects are called \emph{centre pieces} for $A$.    
The \emph{monoidal centre $\cntrZ(A)$ of $A$} is a representing object for $\CP(-,A)\,$:
\[
\funcM\bigl(U, \cntrZ(A)\bigr) \simeq \CP(U,A)\,.
\]
If $\funcM$ is closed as a monoidal bicategory, the pseudo-cosimplicial category is equivalent to
\[
 \UseTips{}\entrymodifiers={+<4mm>!C}\xymatrix{ 
\funcM(U,A)  \ar @<6pt>[r]  \ar @<-6pt>[r]
& \funcM(U,[A,A]) \ar[l]  \ar @<6pt>[r]  \ar[r]   \ar @<-6pt>[r]
& \funcM(U,[A \otimes A,A])\,, }
\]
and $\cntrZ(A)$ is the descent object of the ``Hochschild complex''
\[ 
\UseTips{}\entrymodifiers={+<4mm>!C}\xymatrix{ 
A  \ar @<6pt>[r]  \ar @<-6pt>[r]
& [A,A] \ar[l]  \ar @<6pt>[r]  \ar[r]   \ar @<-6pt>[r]
& [A \otimes A,A]\,. }
\]
 
\begin{Proposition}
$\cntrZ(A)$ is a braided pseudomonoid in $\funcM$.
\end{Proposition}

For a monoidal category $\catC$ and a monoid $A$ in $\catC$, Davydov considers the category whose objects are maps
\[
\zeta : Z \to A
\]
in $\catC$ with $Z$ in $\catZ(\catC)$ and such that the following diagram commutes.
\[
 \xymatrix{Z\otimes A \ar[dd]_-{u_A}   \ar[r]^-{\zeta \otimes 1} & A\otimes A  \ar[rd]^-\mu
\\
&& A
\\
A\otimes Z   \ar[r]_-{1\otimes \zeta} & A\otimes A  \ar[ru]_-\mu
}
\]
A terminal object $i : \zz(A) \to A$ in this category is called the \emph{full centre} of~$A$.   
Davydov proves $\zz(A)$ is a commutative monoid in $\catZ(\catC)$, 
$i$ is a monoid morphism in $\catC$, and $\zz(A) \cong \zz(B)$ for Morita equivalent $A$ and $B$.

\medskip

Consider the 2-category $\Cat_\ast$ of pairs $(\catX,X)$ where $\catX$ is a category and $X$ is an object of $\catX$, 
and where the morphisms $(F,\varphi) : (\catX,X) \to (\catY,Y)$ consist of 
a functor $F : \catX \to \catY$ and a morphism $\varphi : FX \to Y$ in $\catY$.  
This is a symmetric monoidal 2-category under product.

A pseudomonoid in $\Cat_\ast$ is a pair $(\catC,A)$ where $\catC$ is a monoidal category and $A$ is a monoid in $\catC$.  
A braided pseudomonoid is a braided monoidal category $\catZ$ with a commutative monoid $Z$ selected.

\begin{Proposition}
If the full centre $\zz(A)$ of a monoid $A$ in a monoidal category $\catC$ exists (in the sense of Davydov) 
then the centre of $(\catC,A)$ in  $\Cat_\ast$ exists and is given by
\[
\catZ (\catC,A) \simeq \bigl(\catZ(\catC),\zz(A)\bigr)\,.
\]
\end{Proposition}

\begin{center}
--------------------------------------------------------
\end{center}

\appendix

\noindent \small{Centre of Australian Category Theory \\
Macquarie University, NSW 2109 \\
Australia \\
<ross.street@mq.edu.au>}

\end{document}